\documentclass[11pt]{article}

\title{\textbf{Taming the pseudoholomorphic beasts in $\RR\times(S^1\times S^2)$}}

\author{\Large Chris Gerig}

\AtEndDocument{\bigskip{\footnotesize
  \textsc{Department of Mathematics, UC Berkeley, CA 94720, USA} \par  
  \textit{E-mail address:} \texttt{cgerig@berkeley.edu}
}}

\usepackage{hyperref}
\usepackage[initials]{amsrefs}

\usepackage{fullpage}
\usepackage{amssymb,latexsym,amsmath,amsthm,amscd,dsfont,graphicx}
\usepackage{comment}
\usepackage{enumerate}
\usepackage[usenames,dvipsnames,svgnames,table]{xcolor}

\usepackage{epigraph}
\usepackage{etoolbox}

\makeatletter
\patchcmd{\epigraph}{\@epitext{#1}}{\itshape\@epitext{#1}}{}{}
\makeatother

\numberwithin{equation}{section}

\newtheorem{theorem}{Theorem}[section]
\newtheorem{prop}[theorem]{Proposition}

\newtheorem{lemma}[theorem]{Lemma}
\newtheorem{lemma-definition}[theorem]{Lemma-Definition}

\theoremstyle{definition}
\newtheorem{definition}[theorem]{Definition}
\newtheorem{remark}[theorem]{Remark}

\newtheorem{example}[theorem]{Example}

\renewcommand{\SS}{{\mathbb S}}

\newcommand{\CC}{{\mathbb C}}
\newcommand{\QQ}{{\mathbb Q}}

\newcommand{\RR}{{\mathbb R}}

\newcommand{\NN}{{\mathbb N}}
\newcommand{\ZZ}{{\mathbb Z}}

\newcommand{\nN}{{\mathcal N}}
\newcommand{\aA}{{\mathcal A}}

\newcommand{\cC}{{\mathcal C}}
\newcommand{\eE}{{\mathcal E}}
\newcommand{\mM}{{\mathcal M}}

\newcommand{\oO}{{\mathcal O}}
\newcommand{\pP}{{\mathcal P}}
\newcommand{\lL}{{\mathcal L}}

\newcommand{\op}{\operatorname}

\newcommand{\Spinc}{\op{Spin}^c}

\renewcommand{\ker}{\op{Ker}}

\renewcommand{\dbar}{\overline{\partial}}

\newcommand{\PD}{\op{PD}}

\newcommand{\tr}{\op{tr}}
\newcommand{\dv}{d\text{vol}}
\newcommand{\1}{\mathds{1}}
\newcommand{\e}{\varepsilon}

\newcommand{\ind}{\op{ind}}

\newcommand{\fo}{\mathfrak o}

\newcommand{\s}{\mathfrak s}

\newcommand{\Rel}{\op{Rel}_\omega}

\newcommand{\Hfrom}{\widehat{\mathit{HM}}}

    \RequirePackage{rotating}                   
    \def\Hto{%
       \setbox0=\hbox{$\widehat{\mathit{HM}}$}
       \setbox1=\hbox{$\mathit{HM}$}
       \dimen0=1.1\ht0
       \advance\dimen0 by 1.17\ht1
       \smash{\mskip2mu\raise\dimen0\rlap{%
          \begin{turn}{180}
              {$\widehat{\phantom{\mathit{HM}}}$}
           \end{turn}} \mskip-2mu    
                \mathit{HM}
    }{\vphantom{\widehat{\mathit{HM}}}}{}}
    \RequirePackage{rotating}                   
    \def\Cto{%
       \setbox0=\hbox{$\widehat{\mathit{CM}}$}
       \setbox1=\hbox{$\mathit{CM}$}
       \dimen0=1.1\ht0
       \advance\dimen0 by 1.17\ht1
       \smash{\mskip2mu\raise\dimen0\rlap{%
          \begin{turn}{180}
              {$\widehat{\phantom{\mathit{CM}}}$}
           \end{turn}} \mskip-2mu    
                \mathit{CM}
    }{\vphantom{\widehat{\mathit{CM}}}}{}}

\definecolor{blue}{rgb}{0,0,1}
\definecolor{red}{rgb}{1,0,0}
\definecolor{green}{rgb}{0,.7,0}

\begin{document}
\maketitle

\begin{abstract}
For a closed oriented smooth 4-manifold $X$ with $b^2_+(X)>0$, the Seiberg-Witten invariants are well-defined. Taubes' ``SW=Gr'' theorem asserts that if $X$ carries a symplectic form then these invariants are equal to well-defined counts of pseudoholomorphic curves, Taubes' Gromov invariants. In the absence of a symplectic form there are still nontrivial closed self-dual 2-forms which vanish along a disjoint union of circles and are symplectic elsewhere. This paper and its sequel describes well-defined counts of pseudoholomorphic curves in the complement of the zero set of such near-symplectic 2-forms, and it is shown that they recover the Seiberg-Witten invariants over $\ZZ/2\ZZ$. This is an extension of ``SW=Gr'' to non-symplectic 4-manifolds.

The main result of this paper asserts the following. Given a suitable near-symplectic form $\omega$ and tubular neighborhood $\nN$ of its zero set, there are well-defined counts of pseudoholomorphic curves in a completion of the symplectic cobordism $(X-\nN,\omega)$ which are asymptotic to certain Reeb orbits on the ends. They can be packaged together to form ``near-symplectic'' Gromov invariants as a function of spin-c structures on $X$.
\end{abstract}

\section{Introduction}

\subsection{Motivation}

\indent\indent
The Seiberg-Witten invariants, introduced by Witten \cite{Witten:monopoles}, are defined for any closed oriented 4-manifold $X$ with $b^2_+(X)\ge1$. These invariants $SW_X(\s)$ are constructed by counting solutions to the Seiberg-Witten equations associated with a given spin-c structure $\s\in\Spinc(X)$. When $b^2_+(X)>1$ the invariants only depend on the smooth structure of $X$ and a choice of ``homology orientation'' of $X$. When $b^2_+(X)=1$ the invariants also depend on a choice of ``chamber.''

The Gromov invariants, introduced by Taubes \cite{Taubes:counting} who generalized the work of Gromov \cite{Gromov} and Ruan \cite{Ruan:symplectic}, are defined for any closed symplectic 4-manifold $(X,\omega)$. These invariants $Gr_{X,\omega}(A)$ are constructed by counting pseudoholomorphic curves in $X$ which represent a given homology class $A\in H_2(X;\ZZ)$. The invariants only depend on the smooth structure of $X$.

Now, the existence of a symplectic structure implies that $b^2_+(X)>0$, and when $b^2_+(X)=1$ a given symplectic form determines a canonical ``chamber'' with which to define the Seiberg-Witten invariants of $X$. The symplectic form also determines a canonical ``homology orientation'' of $X$ and it induces a canonical identification of $H_2(X;\ZZ)$ with $\Spinc(X)$. Under these identifications it was shown by Taubes \cite{Taubes:SWGrBook} that the Gromov invariants and the Seiberg-Witten invariants are equivalent. But what occurs when $X$ does not have a symplectic structure? This was originally asked by Taubes during 1995 in a paper \cite{Taubes:SWGr} from which we quote the ending:

\begin{center}
{\small{\textit{``It is observed that there are manifolds with nonzero Seiberg-Witten invariants which do not admit symplectic forms. With this understood, one is led to ask whether there is any sort of ``Gromov invariant'' interpretation for the Seiberg-Witten invariants in the nonsymplectic world.''}}}
\end{center}

In the absence of a symplectic form but keeping $b^2_+(X)>0$, given a generic Riemannian metric on $X$ there are still nontrivial closed self-dual 2-forms which vanish transversally along a disjoint union of circles in $X$ and are symplectic elsewhere. These were studied intensively in unpublished work of Karl Luttinger during the early 1990's (see also \cite{LuttingerSimpson}). Such a 2-form $\omega$ is called a \textit{near-symplectic form}, and $X-\omega^{-1}(0)$ is a noncompact symplectic 4-manifold which has an almost complex structure $J$ determined by $\omega$ and the Riemannian metric. It was noticed by Taubes that the estimates on symplectic 4-manifolds which were used to find $J$-holomorphic curves from given Seiberg-Witten solutions could also be used to find $J$-holomorphic curves when $\omega^{-1}(0)\ne\varnothing$.

\begin{theorem}[Taubes \cite{Taubes:SWselfdual}]
\label{thm:Taubes}
If $X$ has a nonzero Seiberg-Witten invariant then there exists at least one $J$-holomorphic curve in $X-\omega^{-1}(0)$, homologically bounding $\omega^{-1}(0)$ in the sense that it has intersection number 1 with every linking 2-sphere of $\omega^{-1}(0)$.
\end{theorem}

This theorem provides evidence that there might be ``Gromov invariants'' for non-symplectic 4-manifolds, and that they might recover the Seiberg-Witten invariants. An elaboration can be found in \cite{Taubes:ICM,Taubes:geom}. But there are difficulties in constructing well-defined counts of such curves because $J$ becomes singular along $\omega^{-1}(0)$. That is, transversality and a Fredholm theory for the moduli spaces of curves in $X-\omega^{-1}(0)$ is hard to establish (see \cite{Taubes:currents,Taubes:moduli}). We can rephrase the problem (see \cite{Taubes:HWZ}), for which it is standard to complete $X-\omega^{-1}(0)$ by attaching symplectization ends and then the $J$-holomorphic curves satisfy certain asymptotic conditions. The boundary $Y$ of a tubular neighborhood of $\omega^{-1}(0)$ is a contact manifold, with contact form induced by $\omega$, and the $J$-holomorphic curves are asymptotic to periodic orbits of the Reeb flow induced by the contact form on $Y$. Models of these curves in the symplectization of $Y=S^1\times S^2$ are studied in \cite{Taubes:beasts,Taubes:puncturedspheres1,Taubes:puncturedspheres2,Mast}, and the aforementioned difficulties are now caused by the contact dynamics: the existence of certain Reeb orbits permits the existence of non-transverse multiply covered curves in the relevant moduli spaces (see Remark~\ref{I<0}).

As will be explained momentarily, we overcome the transversality difficulties by modifying the chosen neighborhood of $\omega^{-1}(0)$ and hence the contact dynamics on each component $S^1\times S^2$ of $Y$. We then pick out the appropriate $J$-holomorphic curves to count using embedded contact homology (ECH), a Floer theory constructed by Hutchings \cite{Hutchings:lectures}. ECH was originally motivated by the desire to find an analog of Taubes' equivalence $SW_X=Gr_{X,\omega}$ in dimension three, granted that a version of Seiberg-Witten Floer homology \cite{KronheimerMrowka:SWF} existed on the gauge theory side. These two homologies are now known to be isomorphic by Taubes \cite{Taubes:ECH=SWF1}, using the same techniques that were used to prove $SW_X=Gr_{X,\omega}$. All of this machinery plays a crucial role in this paper and its sequel \cite{Gerig:Gromov}.

\subsection{Near-symplectic geometry}
\label{Near-symplectic geometry}

\indent\indent
Throughout this paper, $(X,g)$ denotes a closed connected oriented smooth Riemannian 4-manifold with $b^2_+(X)\ge1$, where $b^2_+(X)$ denotes the dimension of any maximal positive-definite subspace of $H^2(X;\RR)$ under the intersection form on $X$. In particular, if $X$ is simply connected then the only spaces excluded by this assumption on $b^2_+(X)$ are 4-spheres and connected sums of $\overline{\CC P^2}$, and if $X=S^1\times M$ then the only spaces excluded are rational homology 3-spheres $M$.

Let $\omega\in\Omega^2(X;\RR)$ be a nontrivial closed self-dual (hence harmonic) 2-form. These always exist by Hodge theory, and in fact, the set of such 2-forms determines a subspace $H^2_+(X;\RR)\subset H^2(X;\RR)$ for which $b^2_+(X)=\dim H^2_+(X;\RR)$. A nice property of $\omega$ is that the complement of its zero set $Z:=\omega^{-1}(0)$ is symplectic,
$$\omega\wedge\omega=\omega\wedge\ast\omega=|\omega|^2\dv_g$$
but we cannot always expect $Z=\varnothing$, i.e. for $X$ to be symplectic. For starters, if a symplectic form existed then an almost complex structure could be built from it and the metric, forcing $1-b^1(X)+b^2_+(X)$ to have even parity by characteristic class theory. What we do know in general is that $Z$ cannot contain any open subset of $X$ \cite{DonaldsonKronheimer}*{Corollary 4.3.23}. 

For generic metrics, Honda \cite{Honda:transversality}*{Theorem 1.1} and LeBrun \cite{LeBrun:Yamabe}*{Proposition 1} have shown that there exist nontrivial closed self-dual 2-forms that vanish transversally as sections of the self-dual 3-plane subbundle $\bigwedge_+^2T^\ast X$ of $\bigwedge^2T^\ast X\to X$. Such 2-forms are examples of

\begin{definition}
A closed 2-form $\omega:X\to\bigwedge^2T^\ast X$ is \textit{near-symplectic} if for all points $x\in X$ either $\omega^2(x)>0$, or $\omega(x)=0$ and the rank of the gradient $\nabla\omega_x:T_xX\to\bigwedge^2T_x^\ast X$ is three.
\end{definition}

It follows from \cite{Auroux:singular}*{Proposition 1} that we can always find a metric for which a given near-symplectic form is self-dual, and we assume throughout this paper that such metrics have been chosen. It follows from the definition that $Z$ consists of a finite disjoint union of smooth embedded circles \cite{Perutz:near}*{Lemma 1.2}, and $Z$ is null-homologous while the individual zero-circles need not be \cite{Perutz:thesis}*{Proposition 1.1.26}. Moreover, the zero-circles are not all the same but come in two ``types'' depending on the behavior of $\omega$ near them:

Let $N_Z$ denote the normal bundle of $Z\subset X$, identified with the orthogonal complement $TZ^\perp$ such that $TX=TZ\oplus N_Z$. The gradient $\nabla\omega$ defines a vector bundle isomorphism $N_Z\to\bigwedge_+^2T^\ast X|_Z$, and we orient $N_Z$ such that $\nabla\omega$ is orientation-reversing. The orientation of $TX$ orients $\bigwedge_+^2T^\ast X$, so $TZ$ is subsequently oriented. As described in \cite{Taubes:currents,Perutz:near,Honda:local}, $\omega$ determines a particular subbundle decomposition
$$N_Z=L_Z\oplus L_Z^\perp$$
where $L_Z\to Z$ is a rank one line bundle. Explicitly, $L_Z$ is the negative-definite subspace with respect to the induced quadratic form on $N_Z$,
$$v\mapsto\langle\nabla_v\omega(\partial_0,\cdot),v\rangle$$
where $\partial_0\in\Gamma(TZ)$ is the unit-length oriented vector field. A zero-circle of $Z$ is called \textit{untwisted} if $L_Z$ restricted to that zero-circle is orientable, and \textit{twisted} otherwise.\footnote{The literature uses the terminology ``even and odd'' as well as ``orientable and non-orientable'' to describe the zero-circles, but the author has found this to be very confusing.}

By work of Luttinger, any given pair $(\omega,g)$ of closed self-dual 2-form and Riemannian metric can be modified so that $Z$ has any positive number of components (see \cite{Perutz:near,Taubes:Luttinger}), but as noted by Gompf (see \cite{Perutz:near}*{Theorem 1.8}), the number of untwisted zero-circles must have the same parity as that of $1-b^1(X)+b^2_+(X)$. We can use these modifications to get rid of all twisted zero-circles (see \cite{Perutz:near}*{Remark 1.9}), and explicit constructions of near-symplectic forms having only untwisted zero-circles are given in \cite{GayKirby,Scott:thesis} in terms of handlebody decompositions.

\begin{example}
The ``symplectic geometer's dream'' would be that $(\CC P^2,\omega_\text{FS})\#(\CC P^2,\omega_\text{FS})$ is symplectic, but this is not the case. The failure to symplectically glue the Fubini-Study 2-forms across the connect-sum region manifests itself in a near-symplectic form with a single untwisted zero-circle.
\end{example}

\subsection{Main results}
\label{Main results}

\indent\indent
Let $\omega$ be a near-symplectic form on $(X,g)$ whose zero set $Z$ has $N\ge0$ components, all of which are untwisted zero-circles. As we just mentioned, this can always be arranged.

\begin{remark}
The assumption that $Z$ has only untwisted zero-circles can be weakened in this paper to some extent. As will be clarified in the appendix, we may assume $Z$ to also have twisted zero-circles which are \textit{non-contractible} in $X$.
\end{remark}

Let $\nN$ denote a union of arbitrarily small tubular neighborhoods of the components of $Z\subset X$. Using Moser-type results of Honda \cite{Honda:local} (see also \cite{Taubes:HWZ}*{\S2e}), $\nN$ can be chosen so that the complement
$$(X_0,\omega):=(X-\nN,\omega|_{X-\nN})$$
is a symplectic manifold with contact-type boundary, where each boundary component is a copy of $(S^1\times S^2,e^{-1}\lambda_0)$. Here, $\lambda_0$ is an overtwisted contact form which is described in Section~\ref{Taubes' contact form}, and studied intensively by Taubes in order to characterize the pseudoholomorphic ``beasts'' living in the symplectization $\RR\times(S^1\times S^2)$.

\begin{remark}
The orientation of $S^1\times S^2$ as a contact 3-manifold disagrees with the orientation of $S^1\times S^2$ as a boundary component of $X_0$, so each boundary component of $X_0$ is \textit{concave} (if the orientations agreed then it would be a \textit{convex} boundary component). This is consistent with a well-known result of Eliashberg: an overtwisted contact 3-manifold cannot be the convex boundary of a symplectic 4-manifold, i.e. there are no symplectic fillings of this contact 3-manifold. But note that $(\nN,\omega)$ is a near-symplectic filling.
\end{remark}

There is a canonical spin-c structure $\s_\omega$ on $X_0$ whose positive spinor bundle is $\SS_+=\underline\CC\oplus K^{-1}$, where $\underline\CC\to X_0$ denotes the trivial complex line bundle and $K$ is the canonical bundle of $(X_0,J)$ for any chosen $\omega$-compatible almost complex structure $J$. Any other spin-c structure on $X_0$ differs from $\s_\omega$ by tensoring with a complex line bundle on $X_0$. It follows from Taubes' work on near-symplectic geometry that $\omega$ also induces a canonical identification of $\Spinc(X)$ with the set
$$\Rel(X):=\Big\{ A\in H_2(X_0,\partial X_0;\ZZ)\;\big\rvert\;\partial A=\1\in H_1(\partial X_0;\ZZ)\Big\}$$
where $\1$ is the oriented generator on each component (the orientation conventions are specified in Section~\ref{Taubes' contact form}). This correspondence is given by restricting $\s\in\Spinc(X)$ to $X_0$, which differs from $\s_\omega$ by a unique complex line bundle $L_\s\to X_0$, and taking the Poincar\'e-Lefschetz dual of $c_1(L_\s)$. Denote the resulting class by $A_\s\in\Rel(X)$.

Of relevance to the Seiberg-Witten invariants of $X$ will be $J$-holomorphic curves in a completion of $(X_0,\omega)$ which represent classes in $\Rel(X)$. The completion $(\overline X,\omega)$ of $(X_0,\omega)$ is obtained by attaching cylindrical ends to the components of $\partial X_0$ and extending $\omega$ over them (see Section~\ref{Curves with}). In order to obtain well-defined counts we need all curves to be transverse. Unfortunately, there are special Reeb orbits in $(S^1\times S^2,\lambda_0)$ for which there exist non-transverse multiply covered curves in $\overline X$ asymptotic to multiple covers of those orbits.\footnote{The problem is worse: these non-transverse curves also have negative ECH index (see Remark~\ref{I<0}).} What can we do?

A given class $A\in\Rel(X)$ determines an upper bound $\rho(A)$ of the ``lengths'' of the possible Reeb orbits which the relevant $J$-holomorphic curves are asymptotic to, so there are only finitely many orbits which permit the existence of non-transverse curves. The idea now is to go back and choose a different tubular neighborhood $\nN$, determining a different contact form on $S^1\times S^2$ having the same contact structure $\xi_0$ as $\lambda_0$. In particular, we search for a particular contact form whose Reeb orbits of ``lengths'' less than $\rho(A)$ do not permit the existence of non-transverse curves representing $A$. This is the content of the following lemma, proved in Section~\ref{Changing Taubes' contact form}.

\begin{lemma}
\label{lem:nbhdintro}
For $A\in\Rel(X)$, there exists a tubular neighborhood $\nN$ of $Z$ in $X$ such that the symplectic cobordism $(X-\nN,\omega)$ has contact-type boundary, where each boundary component is a copy of $(S^1\times S^2,\lambda_A)$. Here, $\lambda_A$ denotes a nondegenerate overtwisted contact form with contact structure $\xi_0$ such that its Reeb orbits of symplectic action less than $\rho(A)$ are $\rho(A)$-flat and are either positive hyperbolic or $\rho(A)$-positive elliptic.
\end{lemma}

The ``$\rho(A)$-positive'' condition on the elliptic orbits is a key ingredient to preventing the existence of (negative ECH index) non-transverse curves that represent $A$, and this was observed by Hutchings \cite{Hutchings:beyond} for more general symplectic cobordisms. The ``$\rho(A)$-flatness'' condition on $\lambda_A$ is needed in order to identify our counts of curves with the Seiberg-Witten invariants, as Taubes \cite{Taubes:ECH=SWF1} did for the isomorphism between ECH and Seiberg-Witten Floer cohomology.

\bigskip
Another issue is the possible existence of multiply covered holomorphic exceptional spheres, which is also an issue in Taubes' Gromov invariants for closed symplectic 4-manifolds. Recall that an \textit{exceptional sphere} in $X$ is an embedded smooth sphere of self-intersection $-1$, and $X$ is \textit{minimal} if there are no exceptional spheres. Let
$$\eE_\omega\subset H_2(X_0,\partial X_0;\ZZ)$$
denote the set of classes represented by \textit{symplectic} exceptional spheres in $X_0$. The issue with multiply covered spheres will be discussed in Section~\ref{ECH cobordism maps}, and it is avoided in the following two scenarios: Either $X$ is minimal, or $E\cdot A\ge-1$ for every $E\in\eE_\omega$.

\bigskip
The relevant counts of curves will a priori depend on a suitably generic choice of $J$ and they will be packaged together as elements in $ECH_\ast(-\partial X_0,\xi_0)$. Roughly speaking, the ECH chain complex over $\ZZ$ is generated by ``orbit sets'' which are finite sets of pairs $(\gamma,m)$ of Reeb orbits $\gamma$ in $(-\partial X_0,\xi_0)$ and multiplicities $m\in\NN$. To define invariants of $X$, we will first count $J$-holomorphic curves in $\overline X$ to obtain a cycle in the ECH chain complex: the asymptotics of each curve define a generator $\Theta$ of the ECH chain complex, and the integer coefficient attached to each such $\Theta$ is an integrally weighted count of the curves which are asymptotic to $\Theta$. Specifically,

\begin{theorem}
\label{main theorem}
Let $(X,\omega,A)$ be as specified above, such that $E\cdot A\ge-1$ for every symplectic exceptional class $E\in\eE_\omega$. Fix a nonnegative integer $I$ and an ordered set of equivalence classes $[\bar\eta]:=\big\lbrace[\eta_1],\ldots,[\eta_p]\big\rbrace\subset H_1(X;\ZZ)/\operatorname{Torsion}$, where $0\le p\le I$ such that $I-p$ is even. For suitably generic $J$ on $\overline X$, there is a well-defined element
$$Gr^I_{X,\omega,J}(A,[\bar\eta])\in ECH_\ast(-\partial X_0,\xi_0)$$
concentrated in a single grading $g(A,I)$. This element is given by integrally weighted counts of (disjoint and possibly multiply covered) $J$-holomorphic curves in $(\overline X,\omega)$ which satisfy the following: They are asymptotic to Reeb orbits with respect to $\lambda_A$ whose total homology class in each $S^1\times S^2$ component is the oriented generator; they represent the relative class $A$; they pass through an a priori chosen set of $\frac{1}{2}(I-p)$ base points in $X_0$; they pass through an a priori chosen ordered set of $p$ disjoint oriented loops in $X_0$ which represent $[\bar\eta]$; and they have ECH index $I$.
\end{theorem}

Here, the ``ECH index'' may be viewed as a formal dimension of the relevant moduli space of $J$-holomorphic currents, while the point/loop constraints carve out a zero-dimensional subset of the moduli space. These curves are explicitly specified in Proposition~\ref{prop:moduli}.

\bigskip
Now, a given $\s\in\Spinc(X)$ determines not only the class $A_\s\in\Rel(X)$ but also an integer
$$d(\s):=\frac{1}{4}\left(c_1(\s)^2-2\chi(X)-3\sigma(X)\right)$$
where $c_1(\s)$ denotes the first Chern class of the spin-c structure's positive spinor bundle, $\chi(X)$ denotes the Euler characteristic of $X$, and $\sigma(X)$ denotes the signature of $X$. This integer $d(\s)$ is the formal dimension of the moduli space of Seiberg-Witten solutions on $X$ with respect to $\s$. 

Using ideas from Seiberg-Witten theory,\footnote{We may also be able to establish this fact using $J$-holomorphic curves alone.} we will show in \cite{Gerig:Gromov} that $g(A_\s,d(\s))$ is the lowest grading for which $ECH_\ast(-\partial X_0,\xi_0)$ is nonzero. After choosing an ordering of the components of $Z$, there is then an identification of the group $ECH_{g(A_\s,d(\s))}(\partial X_0,\xi_0)$ with $\ZZ$ (see Proposition~\ref{prop:ECHSWF}, Remark~\ref{ordering}, and Lemma~\ref{lemma:gradings}).

\begin{definition}
\label{defn:Gr}
Fix an ordering of the zero-circles of $\omega$. The \textit{near-symplectic Gromov invariants}
$$Gr_{X,\omega,J}:\Spinc(X)\to\Lambda^\ast H^1(X;\ZZ)$$
are defined as follows. If $E\cdot A_\s<-1$ for some $E\in\eE_\omega$ then $Gr_{X,\omega,J}(\s)=0$. Assume now that $E\cdot A_\s\ge-1$ for every $E\in\eE_\omega$ (which is a vacuous statement when $X$ is minimal). The component of $Gr_{X,\omega,J}(\s)$ in $\Lambda^pH^1(X;\ZZ)$, for $p\le d(\s)$ such that $d(\s)-p$ is even, is
$$[\eta_1]\wedge\cdots\wedge[\eta_p]\mapsto Gr^{d(\s)}_{X,\omega,J}(A_\s,[\bar\eta])\in\ZZ$$
and it is defined to be zero for all other integers $p$.
\end{definition}

These invariants currently depend on the choice of $\omega$ and $J$, though we expect to establish invariance by varying $(\omega,J)$ and analyzing the resulting moduli spaces of curves. With $\ZZ/2\ZZ$ coefficients and $X$ minimal, we will show in \cite{Gerig:Gromov} that they recover the Seiberg-Witten invariants of $X$, hence are smooth invariants of $X$ alone.\footnote{There is also a version of this result when $X$ is not minimal (see \cite{Gerig:Gromov}).}

\begin{theorem}[Gerig \cite{Gerig:Gromov}]
\label{thm:SWGr}
Let $(X,\omega)$ be as specified above, with $X$ minimal. Given $\s\in\Spinc(X)$,
$$Gr_{X,\omega}(\s)=SW_X(\s)\in\Lambda^\ast H^1(X;\ZZ)\otimes\ZZ/2\ZZ$$
where $\omega$ determines the chamber for defining the Seiberg-Witten invariants when $b^2_+(X)=1$.
\end{theorem}

\begin{remark}
We expect that Theorem~\ref{thm:SWGr} also holds with $\ZZ$ coefficients, similarly to how Taubes' isomorphisms (between embedded contact homology and a version of Seiberg-Witten Floer homology) with $\ZZ/2\ZZ$ coefficients can be lifted to $\ZZ$ coefficients (see \cite{Taubes:ECH=SWF3}*{\S3.b}). Note that the definition of $SW_X$ depends on a choice of homology orientation of $X$, while the definition of $Gr_{X,\omega}$ depends on a choice of ordering of the zero-circles of $\omega$. We expect that $\omega$, with a fixed ordering of its zero-circles, canonically determines a homology orientation of $X$ (see \cite{Gerig:Gromov} for an elaboration).
\end{remark}

Definition~\ref{defn:Gr} is an extension of Taubes' Gromov invariants \cite{Taubes:counting}, which were constructed in the setting $Z=\varnothing$. To see how Taubes' invariants fit into the framework of Theorem~\ref{main theorem}, we first set $\partial X_0=\varnothing$ so that $(\overline X,\omega)=(X,\omega)$ is a closed symplectic 4-manifold. We then make the identification $ECH_0(\varnothing,0)\cong\ZZ$ whose generator is the empty set of orbits. For a given $\s\in\Spinc(X)$, the corresponding absolute class $A_\s\in H_2(X;\ZZ)$ determines the ECH index
$$I(A_\s):=c_1(TX)\cdot A_\s+A_\s\cdot A_\s\in2\ZZ$$
which equals $d(\s)$, and Taubes' invariants may be interpreted as
$$Gr^{I(A_\s)}_{X,\omega,J}(A_\s,[\bar\eta])\in ECH_0(\varnothing,0)$$
since all $J$-holomorphic curves have no punctures.

\begin{remark}
If $X$ is a symplectic manifold and $\omega_0$ is a near-symplectic form, we may consider a deformation of $\omega_0$ to a symplectic form $\omega_1$ through near-symplectic forms (except at a finite number of times along the deformation). The near-symplectic Gromov invariants of $(X,\omega_0)$ must equal Taubes' Gromov invariants of $(X,\omega_1)$ in light of Theorem~\ref{thm:SWGr} and \cite{Taubes:SWGr}*{Theorem 4.1}. The relation between the respective pseudoholomorphic curves is not pursued in this paper, but we leave the reader with a question: When a zero-circle shrinks and dies, does a holomorphic plane decrease its symplectic area and vanish, or does it close up into a sphere, or else?
\end{remark}

\begin{remark}
With respect to Theorem~\ref{thm:Taubes}, we expect that the near-symplectic Gromov invariants are related to counts of $J$-holomorphic curves in $X-Z$ by shrinking $\nN$ to its core $Z$. These curves would homologically bound $Z$ in the sense that they have algebraic intersection number $1$ with every linking 2-sphere of $Z$. A curve may also ``pinch off'' along $Z$ in the sense that a portion of the curve forms a multi-sheeted cone in a small neighborhood of a point on $Z$. These properties correspond to the following facts: a relevant $J$-holomorphic curve in $\overline{X-\nN}$ is asymptotic to an orbit set on each $S^1\times S^2$ whose total homology class is the oriented generator of $H_1(S^1\times S^2;\ZZ)$, and some orbits in that set may be contractible. This is in agreement with Taubes' study \cite{Taubes:currents} of the structure of such curves in the vicinity of $Z$.
\end{remark}

\subsection{Brief outline}

\indent\indent
What follows is an outline of the remainder of the paper. Section~\ref{Review of pseudoholomorphic curve theory} consists of a review of ECH, with a clarification in Section~\ref{L-flat approximations} of ``$L$-flat approximations'' which were defined and used in Taubes' isomorphisms \cite{Taubes:ECH=SWF1}. Section~\ref{Taubes' contact form} consists of the relevant facts of the contact 3-manifold $(S^1\times S^2,\lambda_0$). Section~\ref{Changing Taubes' contact form} specifies how the overtwisted contact form $\lambda_0$ is to be modified, through a particular change in the tubular neighborhood $\nN$ of $Z$. Section~\ref{ECH cobordism maps} clarifies the a priori difficulties with constructing well-defined counts of curves in $X_0=X-\nN$; in particular, we have to separate the cases where the relative classes $A\in\Rel(X)$ may be represented by multiply covered exceptional spheres. Section~\ref{When there are no multiply covered exceptional spheres} specifies the relevant $J$-holomorphic curves when there are no multiply covered exceptional spheres. Section~\ref{Orientations and weights} specifies the integer weights that are assigned to the $J$-holomorphic curves for the definition of $Gr^I_{X,\omega,J}(A,[\bar\eta])$. Section~\ref{Equations for chain maps and chain homotopies} establishes the independence of choices of base points and loops that are used to define the relevant moduli spaces of $J$-holomorphic curves. Section~\ref{Gradings} clarifies what $Gr^I_{X,\omega,J}(A,[\bar\eta])$ looks like as a class in ECH. Lastly, the appendix clarifies what changes are to be made in this paper when $\omega$ is chosen to have twisted zero-circles.

\subsection*{Acknowledgements}

\indent\indent
The author is indebted to Michael Hutchings and Clifford Taubes -- their guidance and ideas were valuable, crucial, and appreciated. The author also thanks Tomasz Mrowka for helping to understand certain aspects of monopole Floer homology. This paper forms part of the author's Ph.D. thesis. The author was partially supported by NSF grants DMS-1406312, DMS-1344991, DMS-0943745. The author thanks Harvard University for its hospitality during a part of Fall 2015.

\section{Review of pseudoholomorphic curve theory}
\label{Review of pseudoholomorphic curve theory}

\indent\indent
The point of this section is to introduce most of the terminology and notations that appear in the later sections. Further information and more complete details are found in \cite{Hutchings:lectures}.

\subsection{Orbits with $\lambda$}
\label{Orbits with}

\indent\indent
Let $(Y,\lambda)$ be a closed contact 3-manifold, oriented by $\lambda\wedge d\lambda>0$. Let $\xi=\ker\lambda$ be the contact structure, oriented by $d\lambda$. Equivalently, $\xi$ is co-oriented by the \textit{Reeb vector field} $R$ determined by $d\lambda(R,\cdot)=0$ and $\lambda(R)=1$. A Reeb orbit is a map $\gamma:\RR/T\ZZ\to Y$ for some $T>0$ with $\gamma'(t)=R(\gamma(t))$, modulo reparametrization, which is necessarily an $m$-fold cover of an embedded Reeb orbit for some $m\ge1$. A given Reeb orbit is \textit{nondegenerate} if the linearization of the Reeb flow around it does not have 1 as an eigenvalue, in which case the eigenvalues are either on the unit circle (such $\gamma$ are \textit{elliptic}) or on the real axis (such $\gamma$ are \textit{hyperbolic}). Assume from now on that $\lambda$ is \textit{nondegenerate}, i.e. all Reeb orbits are nondegenerate, which is a generic property of contact forms. 

An \textit{orbit set} is a finite set of pairs $\Theta=\lbrace(\Theta_i,m_i)\rbrace$ where the $\Theta_i$ are distinct embedded Reeb orbits and the $m_i$ are positive integers (this set may be empty). An orbit set is $\textit{admissible}$ if $m_i=1$ whenever $\Theta_i$ is hyperbolic. Its symplectic action (or ``length'') is defined by
$$\aA(\Theta):=\sum_im_i\int_{\Theta_i}\lambda\ge0$$
and its homology class is defined by
$$[\Theta]:=\sum_im_i[\Theta_i]\in H_1(Y;\ZZ)$$
For a given $\Gamma\in H_1(Y;\ZZ)$, the ECH chain complex $ECC_\ast(Y,\lambda,J,\Gamma)$ is freely generated over $\ZZ$ by equivalence classes of pairs $(\Theta,\fo)$, where $\Theta$ is an admissible orbit set satisfying $[\Theta]=\Gamma$ and $\fo$ is a choice of ordering of the positive hyperbolic orbits and a $\ZZ/2\ZZ$ choice for each such orbit, such that $(\Theta,\fo)=-(\Theta,\fo')$ if $\fo$ and $\fo'$ differ by an odd permutation. We will suppress the notation of the orientation choices $\fo$. The differential $\partial_\text{ECH}$ will be defined momentarily.

\subsection{Curves with $J$}
\label{Curves with}

\indent\indent
Given two contact manifolds $(Y_\pm,\lambda_\pm)$, possibly disconnected or empty, a \textit{strong symplectic cobordism} from $(Y_+,\lambda_+)$ to $(Y_-,\lambda_-)$ is a compact symplectic manifold $(X,\omega)$ with oriented boundary
$$\partial X=Y_+\sqcup -Y_-$$
such that $\omega|_{Y_\pm}=d\lambda_\pm$.
We can always find neighborhoods $N_\pm$ of $Y_\pm$ in $X$ diffeomorphic to $(-\e,0]\times Y_+$ and $[0,\e)\times Y_-$, such that $\omega|_{N_\pm}=d(e^{\pm s}\lambda_\pm)$ where $s$ denotes the coordinate on $(-\e,0]$. We then glue symplectization ends to $X$ to obtain the \textit{completion}
$$\overline X:=\big((-\infty,0]\times Y_-\big)\cup_{Y_-}X\cup_{Y_+}\big([0,\infty)\times Y_+\big)$$
of $X$, a noncompact symplectic 4-manifold whose symplectic form is also denoted by $\omega$. We will also use the notation $\overline X$ to denote the symplectization $\RR\times Y$ of $(Y,\lambda)$, with $\omega=d(e^s\lambda)$.

An almost complex structure $J$ on a symplectization $(\RR\times Y,d(e^s\lambda))$ is \textit{symplectization-admissible} if it is $\RR$-invariant; $J(\partial_s)=R$; and $J(\xi)\subseteq\xi$ such that $d\lambda(v,Jv)\ge 0$ for $v\in\xi$. An almost complex structure $J$ on the completion $\overline X$ is \textit{cobordism-admissible} if it is $\omega$-compatible on $X$ and agrees with symplectization-admissible almost complex structures on the ends $[0,\infty)\times Y_+$ and $(-\infty,0]\times Y_-$.

\begin{remark}
With respect to the Riemannian metric defined by the symplectic form and the admissible almost complex structure, the symplectic form is a self-dual harmonic 2-form.
\end{remark}

Given a cobordism-admissible $J$ on $\overline X$ and orbit sets $\Theta^+=\lbrace(\Theta^+_i,m^+_i)\rbrace$ in $Y_+$ and $\Theta^-=\lbrace(\Theta^-_j,m^-_j)\rbrace$ in $Y_-$, a \textit{$J$-holomorphic curve $\cC$ in $\overline X$ from $\Theta^+$ to $\Theta^-$} is defined as follows. It is a $J$-holomorphic map $\cC\to \overline X$ whose domain is a possibly disconnected punctured compact Riemann surface, defined up to composition with biholomorphisms of the domain, with positive ends of $\cC$ asymptotic to covers of $\Theta^+_i$ with total multiplicity $m^+_i$, and with negative ends of $\cC$ asymptotic to covers of $\Theta^-_j$ with total multiplicity $m^-_j$ (see \cite{Hutchings:lectures}*{\S3.1}). The moduli space of such curves is denoted by $\mM(\Theta^+,\Theta^-)$, but where two such curves are considered equivalent if they represent the same current\footnote{For example, if $\cC$ is a $d$-fold cover of an embedded curve $C$, then the associated current is the $\RR$-valued functional on $\Omega^2(\overline X)$ given by $\sigma\mapsto d\int_C\sigma$. In particular, all branching data of the cover has been lost.} in $\overline X$, and in the case of a symplectization $\overline X=\RR\times Y$ the equivalence includes translation of the $\RR$-coordinate. An element $\cC\in\mM(\Theta^+,\Theta^-)$ can thus be viewed as a finite set of pairs $\lbrace(C_k,d_k)\rbrace$ or formal sum $\sum d_kC_k$, where the $C_k$ are distinct irreducible somewhere-injective $J$-holomorphic curves and the $d_k$ are positive integers. 

Let $H_2(X,\Theta^+,\Theta^-)$ be the set of relative 2-chains $\Sigma$ in $X$ such that
$$\partial \Sigma=\sum_im^+_i\Theta^+_i-\sum_jm^-_j\Theta^-_j$$
modulo boundaries of 3-chains. It is an affine space over $H_2(X;\ZZ)$, and every curve $\cC$ defines a relative class $[\cC]\in H_2(X,\Theta^+,\Theta^-)$.

\bigskip
A \textit{broken $J$-holomorphic curve (of height $n$) from $\Theta^+$ to $\Theta^-$} is a finite sequence of holomorphic curves $\lbrace\cC_k\rbrace_{1\le k\le n}$ and orbit sets $\lbrace\Theta^\pm_k\rbrace_{1\le k\le n+1}$, such that there exists  an integer $1\le k_0\le n$ such that the following holds:

	$\bullet$ $\lbrace\Theta^+_k\rbrace_{k\ge k_0}$ belong to $(Y_+,\lambda_+)$ and $\lbrace\Theta^-_k\rbrace_{k\le k_0}$ belong to $(Y_-,\lambda_-)$,
	
	$\bullet$ $\Theta^-_1=\Theta^-$ and $\Theta^+_{n+1}=\Theta^+$ and $\Theta^-_k=\Theta^+_{k-1}$ for $k>1$,
	
	$\bullet$ $\cC_k\in\mM(\Theta^+_k,\Theta^-_k)$ with respect to $J|_{\RR\times Y_+}$ for $k> k_0$ (\textit{symplectization levels}),
	
	$\bullet$ $\cC_k\in\mM(\Theta^+_k,\Theta^-_k)$ with respect to $J|_{\RR\times Y_-}$ for $k<k_0$ (\textit{symplectization levels}),
	
	$\bullet$ $\cC_{k_0}\in\mM(\Theta^+_{k_0},\Theta^-_{k_0})$ with respect to $J$ (\textit{cobordism level}),
	 
	$\bullet$ $\cC_k$ is not a union of unbranched covers of $\RR$-invariant cylinders for $k\ne k_0$.
	
\noindent
The moduli space of such broken curves is denoted by $\overline{\mM(\Theta^+,\Theta^-)}$. There is an analogous definition of a \textit{broken $J$-holomorphic current}, with the further requirement that each current $\cC_k$ for $k\ne k_0$ is not a union of $\RR$-invariant cylinders with multiplicities.

There are relevant versions of Gromov compactness for the aforementioned moduli spaces. Any sequence of $J$-holomorphic curves $\lbrace C^\nu\rbrace_{\nu\ge1}\subset\mM(\Theta^+,\Theta^-)$ with fixed genus and uniform energy bound has a subsequence which converges in the sense of \textit{SFT compactness} \cite{SFTcompactness} to a broken $J$-holomorphic curve. Any sequence of $J$-holomorphic currents $\lbrace\cC^\nu\rbrace_{\nu\ge1}\subset\mM(\Theta^+,\Theta^-)$ with uniform energy bound has a subsequence which converges in an appropriate sense to a broken $J$-holomorphic current $(\cC_1,\ldots,\cC_n)\in\overline{\mM(\Theta^+,\Theta^-)}$, such that
$$\sum_{k=1}^n[\cC_k]=[\cC^\nu]\in H_2(X,\Theta^+,\Theta^-)$$
for $\nu$ sufficiently large (see \cite{Taubes:currents}*{Proposition 3.3} and \cite{Hutchings:lectures}*{Lemma 5.11}).

\subsection{ECH}
\label{subsec:ECH}

\indent\indent
Denote by $\tau$ a homotopy class of symplectic trivializations of the restrictions of $\xi_\pm=\ker\lambda_\pm$ to the embedded orbits appearing in the orbit sets $\Theta^\pm$. The \textit{ECH index} of a current, or more generally of a class in $H_2(X,\Theta^+,\Theta^-)$, is given by
$$I(\cC)=c_\tau(\cC)+Q_\tau(\cC)+CZ^I_\tau(\cC)$$
where $c_\tau(\cC)$ denotes the relative first Chern class of $\det T\overline X$ over $\cC$ with respect to $\tau$ (see \cite{Hutchings:lectures}*{\S3.2}), $Q_\tau(\cC)$ denotes a relative self-intersection pairing with respect to $\tau$ (see \cite{Hutchings:lectures}*{\S3.3}), and
$$CZ^I_\tau(\cC)=\sum_i\sum_{k=1}^{m^+_i}CZ_\tau(\Theta^{+k}_i)-\sum_j\sum_{k=1}^{m^-_j}CZ_\tau(\Theta^{-k}_j)$$
is a sum of Conley-Zehnder indices (of covers of orbits in $\Theta^\pm$) with respect to $\tau$. The definition of the Conley-Zehnder index will not be reviewed here (see \cite{Hutchings:lectures}*{\S3.2}), but it is noted that we can adjust $\tau$ over a given embedded orbit $\gamma$ so that
$$CZ_\tau(\gamma^m)=0$$
when $\gamma$ is positive hyperbolic,
$$CZ_\tau(\gamma^m)=m$$
when $\gamma$ is negative hyperbolic, and
$$CZ_\tau(\gamma^m)=2\lfloor m\theta\rfloor+1$$
when $\gamma$ is elliptic. Here, the linearization of the Reeb flow around an elliptic orbit is conjugate to a rotation by angle $2\pi\theta$ with respect to $\tau$, and $\theta\in\RR-\QQ$ is the \textit{rotation number}. The equivalence class of $\theta$ in $\RR/\ZZ$ is the \textit{rotation class} of the elliptic orbit, which does not depend on $\tau$.

\bigskip
The \textit{Fredholm index} of a curve $\cC$, having $i^\text{th}$ positive end asymptotic to $\alpha_i$ with multiplicity $m_i$ and $j^\text{th}$ negative end asymptotic to $\beta_j$ with multiplicity $n_j$, is given by
$$\ind(\cC)=-\chi(\cC)+2c_\tau(\cC)+CZ^{\ind}_\tau(\cC)$$
where
$$CZ^{\ind}_\tau(\cC)=\sum_iCZ_\tau(\alpha_i^{m_i})-\sum_jCZ_\tau(\beta_j^{n_j})$$
If $\cC$ has no multiply covered components then we have the \textit{index inequality}
$$\ind(\cC)\le I(\cC)-2\delta(\cC)$$
where $\delta(\cC)$ denotes an algebraic count of singularities of $\cC$ with positive integer weights (see \cite{Hutchings:lectures}*{\S3.4}). If $J$ is furthermore generic then $\mM(\Theta^+,\Theta^-)$ is an $\ind(\cC)$-dimensional manifold near $\cC$.

\bigskip
Denote by $\mM_I(\Theta^+,\Theta^-)$ the subset of elements in $\mM(\Theta^+,\Theta^-)$ that have ECH index $I$. In a symplectization $\overline X=\RR\times Y$ there is a characterization of currents with low ECH index. That is, if $J$ is generic and $\cC$ is a $J$-holomorphic current in the symplectization $\overline X=\RR\times Y$ then

\bigskip
	$\bullet$ $I(\cC)\ge0$, with equality if and only if $\cC$ is a union of $\RR$-invariant cylinders,
	
	$\bullet$ If $I(\cC)=1$ then $\cC=\cC_0\sqcup C_1$, where $C_1$ is an embedded $\ind=I=1$ curve and $I(\cC_0)=0$.

\bigskip\noindent See \cite{Hutchings:lectures}*{Proposition 3.7} for a proof. Given admissible orbit sets $\Theta^\pm$ of $(Y,\lambda)$, the coefficient $\langle\partial_\text{ECH}\Theta^+,\Theta^-\rangle$ is the signed\footnote{The moduli spaces are coherently oriented in the sense of \cite{HutchingsTaubes:gluing2}*{\S9}.} count of elements in $\mM_1(\Theta^+,\Theta^-)$. If $J$ is generic then $\partial_\text{ECH}$ is well-defined and $\partial^2_\text{ECH}=0$. The resulting homology is independent of the choice of $J$, depends only on $\xi$ and $\Gamma$, and is denoted by $ECH_\ast(Y,\xi,\Gamma)$.

\bigskip
If $Y$ is connected then Taubes constructed a canonical isomorphism of relatively graded modules
\begin{equation}
\label{eqn:Taubes}
ECH_\ast(Y,\xi,\Gamma)\cong\Hfrom^{-\ast}(Y,\s_\xi+\Gamma)
\end{equation}
where $\Hfrom^{-\ast}(\cdot)$ is a version of Seiberg-Witten Floer cohomology defined by Kronheimer and Mrowka \cite{KronheimerMrowka:SWF} and $\s_\xi$ is a certain spin-c structure determined by $\xi$. Moreover, both homologies admit absolute gradings by homotopy classes of oriented 2-plane fields on $Y$ and Taubes' isomorphism preserves these gradings (see \cite{Gardiner:gradings}).

\subsection{Gradings and U-maps}

\indent\indent
Assume $Y$ to be connected in this section. For each $\Gamma$ there is a canonical absolute $\ZZ/2\ZZ$ grading on $ECH_\ast(Y,\xi,\Gamma)$ by the parity of the number of positive hyperbolic Reeb orbits in an admissible orbit set $\Theta$. The total sum
$$ECH_\ast(Y,\xi):=\bigoplus_{\Gamma\in H_1(Y;\ZZ)}ECH_\ast(Y,\xi,\Gamma)$$
has an absolute grading by homotopy classes of oriented 2-plane fields on $Y$ (see \cite{Hutchings:revisited}*{\S 3}), the set of which is denoted by $J(Y)$. This grading of an admissible orbit set $\Theta$ is denoted by $|\Theta|$.

As described in \cite{Hutchings:revisited}*{\S 3}, \cite{KronheimerMrowka:SWF}*{\S 28}, and \cite{Gompf:handlebody}*{\S 4}, there is a well-defined map $J(Y)\to\Spinc(Y)$ with the following properties. If $H^2(Y;\ZZ)$ has no 2-torsion then the Euler class of the given 2-plane field uniquely determines the corresponding spin-c structure. There is a transitive $\ZZ$-action on $J(Y)$ whose orbits correspond to the spin-c structures: If $[\xi]\in J(Y)$ then $[\xi]+n$ is the homotopy class of a 2-plane field which agrees with $\xi$ outside a small ball $B^3\subset Y$ and disagrees with $\xi$ on $B^3$ by a map $(B^3,\partial B^3)\to(SO(3),\lbrace\1\rbrace)$ of degree $2n$.\footnote{This convention is opposite to that used in \cite{KronheimerMrowka:SWF}.} A given orbit $J(Y,\s)$ is freely acted on by $\ZZ$ if and only if the corresponding Euler class is torsion.

With that said, $ECH_\ast(Y,\xi,\Gamma)$ has a relative $\ZZ/d\ZZ$ grading, where $d$ denotes the divisibility of $c_1(\xi)+2\PD(\Gamma)$ in $H^2(Y;\ZZ)/\text{Torsion}$. It is refined by the absolute grading and satisfies
$$|\Theta^+|-|\Theta^-|\equiv I(\Sigma)\mod d$$
for any $\Sigma\in H_2(Y,\Theta^+,\Theta^-)$, thanks to the index ambiguity formula \cite{Hutchings:lectures}*{Equation 3.6}.

\bigskip
Similarly to the degree $-1$ ECH differential, there is a degree $-2$ chain map
$$U_y:ECH_\ast(Y,\xi,\Gamma)\to ECH_{\ast-2}(Y,\xi,\Gamma)$$
that counts ECH index 2 currents passing through an a priori chosen base point $(0,y)\in\RR\times Y$, where $y$ does not lie on any Reeb orbit (see \cite{Hutchings:lectures}*{\S3.8}). Such a current is of the form $\cC_0\sqcup C_2$, where $I(\cC_0)=0$ and $C_2$ is an embedded $\ind(C_1)=I(C_1)=2$ curve passing through $(0,y)$. On the level of homology this $U$-map does not depend on the choice of base point.

\subsection{L-flat approximations}
\label{L-flat approximations}

\indent\indent
The symplectic action induces a filtration on the ECH chain complex. For a positive real number $L$, the $L$-\textit{filtered ECH} is the homology of the subcomplex $ECC_\ast^L(Y,\lambda,J,\Gamma)$ spanned by admissible orbit sets of action less than $L$. The ordinary ECH is recovered by taking the direct limit over $L$, via maps induced by inclusions of the filtered chain complexes.

Let $u:C\to \overline X$ be an immersed connected $J$-holomorphic curve and denote by $N_C$ its normal bundle. The linearization of the $J$-holomorphic equation for $C$ defines its \textit{deformation operator}, a 1st order elliptic differential operator
$$D_C:L^2_1(N_C)\to L^2(T^{0,1}C\otimes N_C),\indent\eta\mapsto\dbar\eta+\nu_C\eta+\mu_C\bar\eta$$
where the appropriate sections $\nu_C\in\Gamma(T^{0,1}C)$ and $\mu_C\in\Gamma(T^{0,1}C\otimes N_C^2)$ are determined by the covariant derivatives of $J$ in directions normal to $C$, $\dbar$ is the d-bar operator arising from the Hermitian structure on $N_C$, and $\bar\eta$ denotes the conjugate of $\eta$ in $N^{-1}_C$. Let $N_C$ and $T^{0,1}C$ be suitably trivialized on an end of $C$ asymptotic to a Reeb orbit $\gamma$. Then $D_C$ is asymptotic (in the sense of \cite{Wendl:automatic}*{\S2}) to the \textit{asymptotic operator} associated with $\gamma$,
$$L_\gamma:L^2_1(\gamma^\ast\xi)\to L^2(\gamma^\ast\xi),\indent\eta\mapsto\frac{i}{2}\partial_t\eta+\nu\eta+\mu\bar\eta$$
and the pair $(\nu_C,\mu_C)$ is asymptotic to the pair $(\nu,\mu)$ over $\gamma$.

For a fixed $L>0$ it will be convenient to modify $\lambda$ and $J$ on small tubular neighborhoods of all Reeb orbits of action less than $L$, in order to relate $J$-holomorphic curves to Seiberg-Witten theory most easily. The desired modifications of $(\lambda,J)$ are called \textit{L-flat approximations}, and were introduced by Taubes in \cite{Taubes:ECH=SWF1}*{Appendix} and \cite{Taubes:ECH=SWF5}*{Proposition 2.5}. They induce isomorphisms on the $L$-filtered ECH chain complex, but for the point of this paper (see Lemma~\ref{lem:nbhd} and Section~\ref{Orientations and weights}) we really only need to know that:

	$\bullet$ The Reeb orbits of action less than $L$ (and their action) are not altered,
	
	$\bullet$ The $C^1$-norm of the difference between the contact forms can be made as small as desired,
	
	$\bullet$ The pair $(\nu,\mu)$ associated with an elliptic orbit with rotation number $\theta$ in a given trivialization\\
	\indent~~ is modified to $(\frac{1}{2}\theta,0)$, so that its asymptotic operator is complex linear.
	
\begin{remark}
The proof of Theorem~\ref{thm:SWGr} in \cite{Gerig:Gromov} will make crucial use of Taubes' isomorphism between ECH and Seiberg-Witten Floer cohomology, and that makes use of $L$-flat approximations. The key fact here is that Taubes' isomorphism actually exists on the $L$-filtered chain complex level, for which $L$-flat orbit sets are in bijection with Seiberg-Witten solutions of ``energy'' less than $2\pi L$.
\end{remark}

\section{Towards a near-symplectic Gromov invariant}

\indent\indent
This section spells out the proof of Theorem~\ref{main theorem}. We introduce the contact form on $S^1\times S^2$ that was studied in the past by Taubes, and then we find a different contact form on $S^1\times S^2$ whose contact dynamics is ``tame'' in a certain sense. This new contact form is better suited for establishing well-defined counts of $J$-holomorphic curves in $X-Z$. We also find a tubular neighborhood $\nN$ of $Z$ in $X$ for which the induced contact dynamics on its boundary $\partial\nN$ is given by our new contact form. Then we construct the relevant moduli spaces of $J$-holomorphic curves in $X-\nN$ which are to be counted, including their integer weights, and from there we define the relevant class in ECH. We retain the same assumptions that are made in Section~\ref{Main results}.

\subsection{Taubes' contact form}
\label{Taubes' contact form}

\indent\indent
The main focus of this paper is $S^1\times S^2$ equipped with \textit{Taubes' contact form}
\begin{equation}
\label{eqn:lambda}
\lambda_0=-(1-3\cos^2\theta)dt-\sqrt{6}\cos\theta\sin^2\theta d\varphi
\end{equation}
for coordinates $(t,\theta,\varphi)\in S^1\times S^2$ such that $0\le t\le 2\pi$ and $0\le\theta\le \pi$ and $0\le \varphi\le2\pi$. In order for $\lambda_0\wedge d\lambda_0$ to be positive, $S^1\times S^2$ is oriented by the 3-form
$$-\sin\theta\,dt\,d\theta\,d\varphi$$
The $S^1$-factor will be oriented by the 1-form $-dt$, and the $S^2$-factor will be oriented by the 2-form $\sin\theta\,d\theta\,d\varphi$. This contact manifold was originally studied in \cite{Taubes:HWZ,Honda:local,Perutz:thesis}, and we now describe some details that will be of use later on.

The Reeb field associated with $\lambda_0$ is
$$\frac{-1}{1+3\cos^4\theta}\left((1-3\cos^2\theta)\partial_t+\sqrt{6}\cos\theta\,\partial_\varphi\right)$$
The closed Reeb orbits live in the constant $\theta=\theta_0$ slices of $S^1\times S^2$ satisfying
\begin{equation}
\label{eqn:angleTaubes}
\theta_0\in\lbrace0,\pi\rbrace\indent\text{or}\indent\frac{\sqrt{6}\cos\theta_0}{1-3\cos^2\theta_0}\in\QQ\cup\lbrace\pm\infty\rbrace
\end{equation}
The two nondegenerate orbits at $\theta_0\in\lbrace0,\pi\rbrace$ are elliptic. They are denoted by $e_0,e_\pi$ and called \textit{the exceptional orbits}. The remaining orbits are degenerate, there being an $S^1$-family of orbits for each such $\theta_0$. In other words,
$$T(\theta_0):=\lbrace\text{constant $\theta=\theta_0$ slice}\rbrace\subset S^1\times S^2$$
is a torus foliated by orbits. This contact form is thus not nondegenerate, but it is ``Morse-Bott'' in the sense of \cite{Bourgeois:thesis}.

\begin{remark}
The contact structure $\xi_0=\ker\lambda_0$ is overtwisted, as pointed out in \cite{Taubes:HWZ}*{\S 2.f} and \cite{Honda:local}*{Proposition 9}. A well-known result of Hofer states that an overtwisted contact 3-manifold must have at least one contractible orbit, and indeed we see that $T(\arccos(\frac{1}{\sqrt 3}))$ and $T(\arccos(-\frac{1}{\sqrt 3}))$ consist of contractible orbits. The remaining orbits are all homologically nontrivial.
\end{remark}

With the contact structure $\xi_0$ oriented by $d\lambda_0$, we compute
$$c_1(\xi_0)=-2\in\ZZ\cong H^2(S^1\times S^2;\ZZ)$$
by using the section $\sin\theta\,\partial_\theta\in\Gamma(\xi_0)$ and noting that the orientation on $\xi_0$ disagrees with the orientation on the $S^2$-factor at $\theta=0$ and $\theta=\pi$. The spin-c structure $\s_\xi$ determined by $\xi_0$ satisfies
$$c_1(\s_\xi+1)=c_1(\s_\xi)+2=c_1(\xi_0)+2=0$$
and so Taubes' isomorphism~\eqref{eqn:Taubes} reads
$$ECH_j(S^1\times S^2,\xi_0,1)\cong \Hfrom^j(S^1\times S^2, \s_\xi+1)$$
where $j\in J(S^1\times S^2,\s_\xi+1)\cong\ZZ$ as $\ZZ$-sets. There is a unique class $j=[\xi_\ast]$ represented by an oriented 2-plane field $\xi_\ast$ on $S^1\times S^2$ which has vanishing Euler class and is invariant under rotations of the $S^1$-factor.

\begin{prop}
\label{prop:ECHSWF}
If $\Gamma\in H_1(S^1\times S^2;\ZZ)$ is not the oriented generator $1$ then
$$ECH_\ast(S^1\times S^2,\xi_0,\Gamma)=0$$
In the remaining case $\Gamma=1$, $ECH_j(S^1\times S^2,\xi_0,1)$ is zero in gradings below $j=[\xi_\ast]$, and for each $n\ge0$
$$ECH_{[\xi_\ast]+n}(S^1\times S^2,\xi_0,1)\cong\ZZ$$
The homotopy class $[\xi_\ast]$ has odd parity under the absolute $\ZZ/2\ZZ$ grading on ECH. After perturbing $\lambda_0$ to a nondegenerate contact form, $[\xi_\ast]$ is generated by a single positive hyperbolic orbit wrapping positively once around the $S^1$-factor in $T(\frac{\pi}{2})=S^1\times\lbrace\text{equator}\rbrace$.
\end{prop}

\begin{proof}
The first two statements are proved in \cite{KronheimerMrowka:SWF}*{\S IX.36} for the relevant version of Seiberg-Witten Floer homology, so they follow for ECH via Taubes' isomorphism~\eqref{eqn:Taubes}. The latter two statements follow from the relevant statements in \cite{Hutchings:T3}*{\S 12.2.1}. This reference uses a ``twisted'' version of ECH that remembers some information about the relative homology classes of the $J$-holomorphic curves in $S^1\times S^2$, but the untwisted version may be obtained via a spectral sequence in \cite{Hutchings:T3}*{\S 8.1}.
\end{proof}

In the upcoming section, $\lambda_0$ will be modified in various ways. We now preemptively analyze the contact form $e^f\lambda_0$ for a given smooth function $f:S^1\times S^2\to\RR$ depending only on the $\theta$ coordinate, whose contact structure is nonetheless $\xi_0$. Such a contact form can be written as
\begin{equation}
\label{eqn:form}
\lambda=a_1(\theta)dt + a_2(\theta)d\varphi
\end{equation}
for some smooth pair
$$a=(a_1,a_2):[0,\pi]\to\RR^2-\lbrace(0,0)\rbrace$$
Let $a\times a':=a_1a_2'-a_2a_1'$, where the tick-mark signifies the derivative with respect to $\theta$. The condition for $\lambda$ to be a positive contact form is then
$$\frac{a\times a'(\theta)}{\sin\theta}<0$$
for all $\theta\in[0,\pi]$. For $\theta\in(0,\pi)$ the Reeb field of $\lambda$ is
$$R=\frac{1}{a\times a'(\theta)}\left(a_2'(\theta)\partial_t-a_1'(\theta)\partial_\varphi\right)$$
and the condition~\eqref{eqn:angleTaubes}, for which $T(\theta_0)\subset S^1\times S^2$ is a torus foliated by orbits, is now given by
\begin{equation}
\label{eqn:angle}
\frac{a_1'(\theta_0)}{a_2'(\theta_0)}\in\QQ\cup\lbrace\pm\infty\rbrace
\end{equation}
Every embedded orbit in $T(\theta_0)$ represents the same class in $H_1(T(\theta_0);\ZZ)$ and they all have the same action $\aA(\theta_0)>0$. There are also two exceptional nondegenerate elliptic orbits at $\theta_0\in\lbrace0,\pi\rbrace$.

\begin{lemma}
\label{lem:rotation}
The exceptional elliptic orbit at $\theta=\theta_0$, for $\theta_0\in\lbrace0,\pi\rbrace$, has rotation class
$$\left(\operatorname{sign}\lim_{\theta\to \theta_0}\frac{-a_2'(\theta)}{\sin\theta\,\cos\theta}\right)\lim_{\theta\to \theta_0}\frac{a_1'(\theta_0)}{a_2'(\theta_0)}\mod1$$
In particular, the rotation class for either exceptional orbit of $\lambda_0$ is $\sqrt{\frac{3}{2}}\mod1$.
\end{lemma}

\begin{proof}
The Reeb field along each exceptional orbit is
$$R=\frac{1}{a_1(\theta_0)}\partial_t$$
Note that $a_1(\theta_0)>0$, because $\lambda_0$ satisfies it and $e^f$ is a positive rescaling. Thus, in a neighborhood of $\theta=\theta_0$ the Reeb flow is in the positive $t$-direction and wraps $-\lim_{\theta\to \theta_0}a_1'(\theta_0)/a_2'(\theta_0)$ times around the $\varphi$-coordinate circle after traversing once around the $t$-coordinate circle. The rotation class of each elliptic orbit is thus
\begin{equation}
\label{eqn:rotation}
\epsilon\cdot\lim_{\theta\to \theta_0}\frac{-a_1'(\theta_0)}{a_2'(\theta_0)}\mod1
\end{equation}
where $\epsilon=\pm1$ depending on whether the $\varphi$-coordinate circle is positively or negatively oriented with respect to the orientation of the contact planes $T_{\theta_0}S^2$ given by $d\lambda$. To determine $\epsilon$, we use the Cartesian coordinates
$$\begin{cases}
(x,y)=(\sin\theta\,\cos\varphi,\sin\theta\,\sin\varphi)\\
dx\,dy=\sin\theta\,\cos\theta\,d\theta\,d\varphi
\end{cases}$$
near each pole of $S^2$. In these coordinates,
$$d\lambda|_{T_{\theta_0}S^2}=\left(\lim_{\theta\to \theta_0}\frac{a_2'(\theta)}{\sin\theta\,\cos\theta}\right)dx\,dy$$
and $\epsilon$ is precisely the sign of this paranthetical expression.
\end{proof}

\subsection{Changing Taubes' contact form}
\label{Changing Taubes' contact form}

\indent\indent
In order to obtain well-defined counts of $J$-holomorphic curves which represent a given class $A\in\Rel(X)$, we will need to ensure a bound on their energy as well as a bound on the symplectic action of their orbit sets. As explained in \cite{Hutchings:fieldtheory}, these bounds are given by a particular quantity $\rho(A)$, defined as follows:

Let $u:\Sigma\to X_0$ be any given smooth map which represents $A$, where $\Sigma$ is a compact oriented smooth surface with boundary and $u(\partial \Sigma)\subset \partial X_0$. Then
\begin{equation}
\rho(A):=\int_\Sigma\omega+\int_{\partial\Sigma}e^{-1}\lambda_0
\end{equation}
This quantity is additive with respect to composition of symplectic cobordisms, and it vanishes on exact symplectic cobordisms (recall that a symplectic cobordism is \textit{exact} if the contact form on the boundary extends as a global primitive 1-form of the symplectic form). Subsequently, $\rho(A)$ does not change if the symplectic cobordism $(X_0,\omega)$ is modified by composing it with an exact symplectic cobordism.

\begin{remark}
We can view $\rho:\Rel(X)\to\RR$ as either a measure of the failure of exactness of $\omega$, or as an energy. In particular, if $(X,\omega)$ were a closed symplectic manifold, i.e. $\partial X_0=\varnothing$ and $A\in H_2(X;\ZZ)$, then $\rho(A)=A\cdot[\omega]$.
\end{remark}

Now, three modifications will be made to $\lambda_0$. First, we will want all Reeb orbits to be nondegenerate in order to define ECH. Second, we will want all Reeb orbits of action less than $\rho(A)$ to be $\rho(A)$-flat in order to relate the $J$-holomorphic curves to Seiberg-Witten theory. Third, we will want the elliptic orbits of action less than $\rho(A)$, especially the exceptional orbits, to be ``$\rho(A)$-positive'' in order to guarantee transversality of the relevant moduli spaces of $J$-holomorphic curves (see Remark~\ref{I<0} below). As defined in \cite{Hutchings:beyond}, the quantifier ``$\rho(A)$-positive'' means the following:

\begin{definition}
Fix $L>0$. Let $\gamma$ be a nondegenerate embedded elliptic orbit with rotation class $\theta\in\RR/\ZZ$ and symplectic action $\aA(\gamma)<L$. Then $\gamma$ is \textit{L-positive} if $\theta\in(0,\aA(\gamma)/L)\mod 1$.
\end{definition}

A key property of any $L$-positive elliptic orbit $\gamma$ is that if $L$ is much greater than $\aA(\gamma)$, then $CZ_\tau(\gamma^m)=1$ for all $m<L/\aA(\gamma)$ and a particular choice of trivialization $\tau$ of $\gamma^\ast\xi_0$.

\bigskip
The next lemma below shows how to modify the Morse-Bott orbits, in the sense of \cite{Bourgeois:thesis} and adapted from \cite{Hutchings:beyond}*{Lemma 5.4}. For a given positive contact form written as~\eqref{eqn:form}, the lemma requires the following technical condition
\begin{equation}
\label{eqn:technical}
a'\times a''(\theta_0)<0
\end{equation}
for all $\theta_0\in(0,\pi)$ that satisfy~\eqref{eqn:angle}. Note that $\lambda_0$ satisfies the technical condition.

\begin{lemma}
\label{lem:Bourgeois}
Suppose the positive contact form $\lambda=a_1(\theta)dt + a_2(\theta)d\varphi$ satisfies the technical condition~\eqref{eqn:technical}. Then for every $L>0$ and sufficiently small $\delta>0$, there exists a perturbation $e^{f_{\delta,L}}\lambda$ of $\lambda$ satisfying the following properties:

	$\bullet$ $f_{\delta,L}\in C^\infty(S^1\times S^2)$ satisfies $||f_{\delta,L}||_{C^0}<\delta$,

	$\bullet$ $e^{f_{\delta,L}}\lambda$ agrees with $\lambda$ near the exceptional orbits at $\theta_0\in\lbrace0,\pi\rbrace$,

	$\bullet$ Each family of orbits in the torus $T(\theta_0)$ with $\aA(\theta_0)<L$ is replaced by a positive hyperbolic\\
	\indent~~~orbit and an $L$-positive elliptic orbit, both of action less than $L$ and within $\delta$ of $\aA(\theta_0)$,

	$\bullet$ $e^{f_{\delta,L}}\lambda$ has no other embedded orbits of action less than $L$.
\end{lemma}

\begin{proof}
The function $f_{\delta,L}$ is given by Bourgeois' perturbation \cite{Bourgeois:thesis} of $\lambda$, which breaks up each $T(\theta_0)$ into two embedded nondegenerate orbits of action slightly less than $\aA(\theta_0)$ in addition to orbits of action greater than $L$. Namely, there is a positive hyperbolic orbit and an elliptic orbit $e_{\theta_0}$, both representing the same class in $H_1(T(\theta_0);\ZZ)$. For sufficiently small perturbations there cannot exist other orbits of action less than $L$, otherwise we would find a sequence $\lbrace(\gamma_k,\delta_k)\rbrace_{k\in\NN}$ of such orbits of uniformly bounded action $L$ and perturbations $\delta_k\to0$ for which a subsequence converges to one of the original degenerate orbits (by the Arzel\`a-Ascoli theorem), yielding a contradiction.

It remains to compute the rotation class of the elliptic orbit created from each Morse-Bott family. Let $a^\perp:=a_2(\theta)\partial_t-a_1(\theta)\partial_\varphi$. The basis $\langle\partial_\theta,a^\perp\rangle$ defines a trivialization $\tau$ of the contact structure $\xi$ over $S^1\times\left(S^2-\lbrace\text{poles}\rbrace\right)$, since
$$d\lambda(\partial_\theta,a^\perp)=-a\times a'(\theta)>0$$
for $0<\theta<\pi$.
We then compute the Lie derivatives of the Reeb field $R$,
$$\lL_{\partial_\theta}R=-\frac{a'\times a''}{(a\times a')^2}a^\perp,\indent \lL_{a^\perp}R=0$$
to see that the linearized Reeb flow along $T(\theta_0)$ induces the linearized return map
$$P_{T(\theta_0)}:=\1+\begin{pmatrix}
0&0\\ r(\theta_0)\aA(\theta_0)&0
\end{pmatrix}$$
on $\xi$ in the chosen basis, where
\begin{equation}
\label{eqn:return}
r:=-\frac{a'\times a''}{(a\times a')^2}
\end{equation}
The linearized return map along $e_{\theta_0}$, denoted $P_{\theta_0}$, is a perturbation of $P_{T(\theta_0)}$. To clarify the dependence of this perturbation on $\delta$, we first note that Bourgeois' factor $e^{f_{\delta,L}}$ takes the form $(1+\delta \tilde f)$ for some function $\tilde f$ that is compactly supported in neighborhoods of the tori $T(\theta_0)$ and does not depend on $\delta$. We then claim that the perturbed Reeb field takes the form $R+\oO(\delta)R_1$, as in \cite{Bourgeois:thesis}*{Lemma 2.3}, where $R_1$ is a vector field that may be expressed in terms of $\lbrace\tilde f, \partial\tilde f, a, a'\rbrace$. Assume this claim for the moment; it will be demonstrated at the very end. So the perturbed linearized return map along $e_{\theta_0}$ with respect to $\tau$ is
\begin{equation}
\label{eqn:perturb1}
P_{\theta_0}=P_{T(\theta_0)}+\oO(\delta)\begin{pmatrix}
P_{11}&P_{12}\\ P_{21}&P_{22}
\end{pmatrix}
\end{equation}
where each entry $P_{ij}$ may be expressed in terms of $\lbrace\tilde f, \partial\tilde f, \partial^2\tilde f, a, a',a''\rbrace$ along $e_{\theta_0}$. We use this bound on the entries of the matrix $P_{\theta_0}-P_{T(\theta_0)}$ to analyze the rotation number now.

Since $P_{\theta_0}$ is a symplectic matrix which is conjugate in $SL_2(\RR)$ to a rotation matrix, we write
\begin{equation}
\label{eqn:perturb2}
P_{\theta_0}=\begin{pmatrix}
u_{11}&u_{12}\\ u_{21}&u_{22}
\end{pmatrix}\begin{pmatrix}
\cos\phi&-\sin\phi\\ \sin\phi&\cos\phi
\end{pmatrix}\begin{pmatrix}
u_{11}&u_{12}\\ u_{21}&u_{22}
\end{pmatrix}^{-1}
\end{equation}
where $u_{11}u_{22}-u_{12}u_{21}=1$ and the rotation number is denoted by $\phi$. Thus,
$$|2\cos\phi-2|=|\tr(P_{\theta_0}-P_{T(\theta_0)})|=\oO(\delta)$$
and so $\phi$ can be made arbitrarily small by choosing $\delta$ sufficiently small. Moreover, the bottom-left entry of $P_{\theta_0}$ in~\eqref{eqn:perturb1} and~\eqref{eqn:perturb2} is
$$(u_{21}^2+u_{22}^2)\sin\phi=r(\theta_0)\aA(\theta_0)+\oO(\delta)$$
and so the sign of $\phi$ is the same as the sign of $r(\theta_0)$ for $\delta$ sufficiently small. This precisely means the rotation number of $e_{\theta_0}$ has the opposite sign of $a'\times a''(\theta_0)$, using the description~\eqref{eqn:return}, so it follows from the technical condition~\eqref{eqn:technical} that each $e_{\theta_0}$ is $L$-positive.

The remaining loose thread is the claimed description of the perturbed Reeb field. In the basis $\langle R,\partial_\theta,a^\perp\rangle$ of the tangent bundle of $S^1\times\left(S^2-\lbrace\text{poles}\rbrace\right)$ consider the ansatz
$$R_\text{perturb}=R+(x_0\cdot R+x_1\cdot\partial_\theta+x_2\cdot a^\perp)$$
The normalization condition $(1+\delta\tilde f)\lambda(R_\text{perturb})=1$ forces
$$x_0=-\frac{\delta}{1+\delta \tilde f}\tilde f$$
because $\lambda(R)=1$ and $\lambda(\partial_\theta)=0$ and $\lambda(a^\perp)=0$. The other defining equation
$$0=d\left((1+\delta\tilde f)\lambda\right)(R_\text{perturb},\cdot)=(1+\delta\tilde f)d\lambda(R_\text{perturb},\cdot)+\delta(d\tilde f\wedge\lambda)(R_\text{perturb},\cdot)$$
splits into $dt$, $d\varphi$, $d\theta$ component equations. Noting that $d\lambda(R,\cdot)=0$, the $d\theta$ component equation forces
$$x_2=\frac{\delta}{(1+\delta\tilde f)^2(a\times a')}\frac{\partial\tilde f}{\partial\theta}$$
and subsequently the $dt$ (or $d\varphi$) component equation forces
$$x_1=\frac{\delta}{(1+\delta\tilde f)^2(a\times a')}(a_1\frac{\partial\tilde f}{\partial\varphi}-a_2\frac{\partial\tilde f}{\partial t})$$
Thus, $R_\text{perturb}$ is as claimed.
\end{proof}

\bigskip
Now we show how to modify the exceptional orbits while preserving the technical condition~\eqref{eqn:technical}.

\begin{lemma}
\label{lem:f}
Given $c\ge0$ and $\e\in(0,\sqrt{\frac{3}{2}}]$, there exists a smooth nonpositive function $f_{\e,c}$ on $S^1\times S^2$ which only depends on the $\theta$ coordinate, such that $e^{f_{\e,c}}\lambda_0$ satisfies the technical condition~\eqref{eqn:technical} and whose exceptional orbits $e_0$ and $e_\pi$ have rotation classes both equal to $\e\mod 1$.
\end{lemma}

\begin{proof}
By Lemma~\ref{lem:rotation}, the rotation classes under consideration are
$$\left(\operatorname{sign}\lim_{\theta\to \theta_0}(3\cos\theta-\frac{1}{\cos\theta})\right)\frac{\frac{\partial f_{\e,c}}{\partial\theta}\frac{3\cos^2\theta-1}{\sin\theta}-6\cos\theta}{\sqrt{6}(1-3\cos^2\theta)-\frac{\partial f_{\e,c}}{\partial\theta}\sqrt{6}\cos\theta\sin\theta}\mod1
$$
for $\theta_0\in\lbrace0,\pi\rbrace$. After setting
$$\frac{\partial f_{\e,c}}{\partial\theta}:=\tilde f_\e(\theta)\sin\theta$$
for a smooth function $\tilde f_\e:[0,\pi]\to\RR$, it follows that the rotation classes of $e_0$ and $e_\pi$ are respectively given by
$$-\frac{1}{\sqrt{6}}\left(\tilde f_\e(\theta)-3\cos\theta\right)\Big\rvert_{\theta=0}\mod1$$
$$\frac{1}{\sqrt{6}}\left(\tilde f_\e(\theta)-3\cos\theta\right)\Big\rvert_{\theta=\pi}\mod1$$
These are both equal to $\e$ upon setting $\tilde f_\e(\theta)=(3-\sqrt{6}\e)\cos\theta$. We then pick the antiderivative of $\tilde f_\e$ using $c$ so that the desired function on $S^1\times S^2$ is
\begin{equation}
\label{eqn:f}
f_{\e,c}(t,\theta,\varphi)=-\frac{3-\sqrt{6}\e}{2}\cos^2\theta-c
\end{equation}
A brute force calculation shows that $e^{f_{\e,c}}\lambda_0$ satisfies the technical condition~\eqref{eqn:technical}.
\end{proof}

\begin{remark}
As a sanity check, if $\varepsilon=\sqrt{\frac{3}{2}}$ and $c=0$ then we recover $\lambda_0$.
\end{remark}

We now move forward and show how exactly $\lambda_0$ is to be modified to prove Lemma~\ref{lem:nbhdintro} in the introduction. Such modifications give us control over the orbits of low symplectic action, at the expense of producing new orbits of high symplectic action with unknown properties. This is sufficient for the purposes of this paper, because for a given class $A\in\Rel(X)$ only the orbit sets of symplectic action less than $\rho(A)$ are relevant to the tentative Gromov invariant.

\begin{lemma}
\label{lem:nbhd}
Suppose $(X,\omega)$ is a near-symplectic 4-manifold such that all components of $Z$ are untwisted zero-circles, and fix $A\in\Rel(X)$. Then there is a choice of neighborhood $\nN$ of $Z$ in $X$ such that $(X-\nN,\omega)$ is a symplectic manifold with contact-type boundary whose boundary components are copies of $(S^1\times S^2,\lambda_A)$. Here, $\lambda_A$ is a nondegenerate contact form with contact structure $\xi_0$ but whose orbits of symplectic action less than $\rho(A)$ are all $\rho(A)$-flat and are either positive hyperbolic or $\rho(A)$-positive elliptic.
\end{lemma}

\begin{proof}
By Lemma~\ref{lem:f}, for any $L$ sufficiently greater than $\rho(A)$ there is a smooth function $f_{\e,c}$ on $S^1\times S^2$ for any $c\ge0$ and sufficiently small choice of $\e$ such that the exceptional orbits of $e^{f_{\e,c}}\lambda_0$ are $L$-positive. By Lemma~\ref{lem:Bourgeois}, we can perturb this new Morse-Bott contact form (\`a la Bourgeois) so that all orbits of action less than $L$ are nondegenerate and $L$-positive when elliptic. As explained in Section~\ref{L-flat approximations}, we can perturb the resulting contact form (\`a la Taubes) so that all orbits of action less than $L$ are furthermore $L$-flat. This new contact form is still degenerate, but it can be perturbed to become nondegenerate without disturbing the existing orbits of action less than $L$ and without introducing other orbits of action less than $L$. Explicitly, let $U\subset S^1\times S^2$ be the union of $L$-flat neighborhoods of the orbits of action less than $L$, which do not contain any other orbits, and use \cite{Bourgeois:survey}*{Lemma 2.2} to make a generic perturbation of the contact form on the complement of $U$.

The final result is a contact form $\lambda_A$ which has the same contact structure as $\lambda_0$, so for $c$ sufficiently large there exists a smooth negative function $F_A\in C^\infty(S^1\times S^2)$ such that the maximum value of $F_A$ over $S^1\times S^2$ is less than $-1$ and
$$\lambda_A=e^{F_A}\lambda_0$$
 
It remains to show that the neighborhood $\nN$ can be chosen so that the contact form on any boundary component of $X-\nN$ is $\lambda_A$, preserving compatibility with $\omega$. Assume without loss of generality that $Z$ is a single circle. We first choose the neighborhood $\nN_\ast$ of $Z$ as specified in Section~\ref{Main results}, but constrained further so that for any fixed $\kappa>1$ there is a neighborhood $\nN_\kappa\subset\nN_\ast$ of $Z$ such that $(\nN_\ast-\nN_\kappa,\omega)$ is symplectomorphic to $\big([-\kappa,-1)\times S^1\times S^2,d(e^s\lambda_0)\big)$, where $s$ denotes the coordinate on $[-\kappa,-1)$. The fact that we may find such a neighborhood will be proved at the end. We choose $\kappa\ge-\inf F_A>1$, because $(\nN_\ast-\nN_\kappa,\omega)$ then contains the contact hypersurface
$$\Big\{(s,x)\in\RR\times(S^1\times S^2)\;\big\rvert\;s=F_A(x)\Big\}$$
with contact form $\lambda_A$. Now our desired neighborhood is
$$\nN:=\nN_\kappa\cup\Big\{(s,x)\in[-\kappa,-1)\times(S^1\times S^2)\;\big\rvert\;s\le F_A(x)\Big\}$$
which is the subset of $\nN_\ast$ ``below the contact hypersurface'' (a schematic is given by Figure~\ref{fig:N}). Note that the quantity $\rho(A)$ has not changed, since $(X-\nN,\omega)$ differs from $(X-\nN_\ast,\omega)$ by composition with an exact symplectic cobordism.

\begin{figure}
    \centering
    \includegraphics[width=6cm]{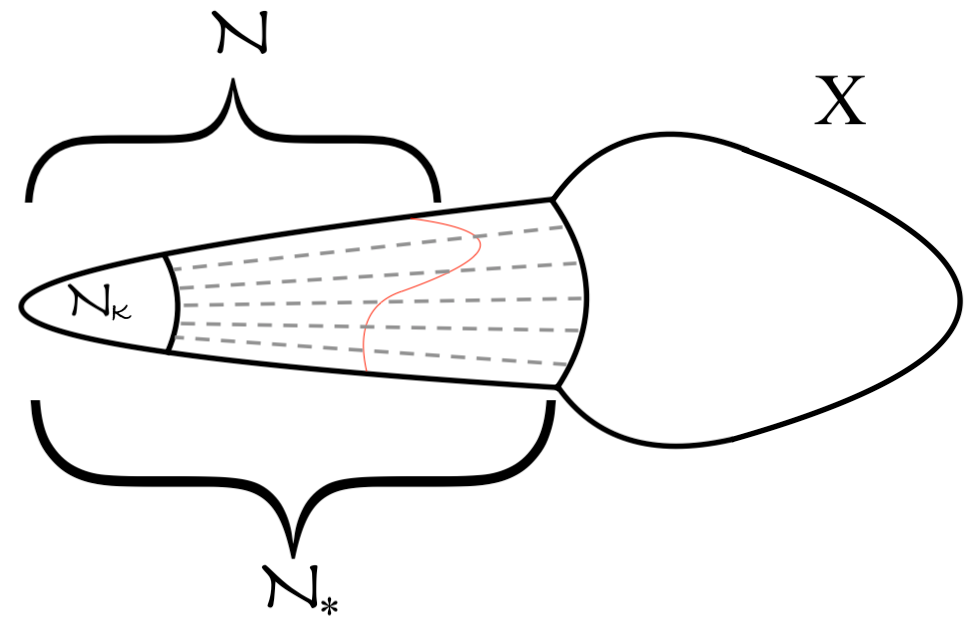}
    \caption{Hypersurface defined by $F_A$ in red}
    \label{fig:N}
\end{figure}

We supply here the construction of $\nN_\ast$, as alluded to by \cite{Taubes:HWZ}*{Lemma 2.3}, and we do it in two (not so different) ways. The first way is to assume that we have already found a neighborhood $\nN_0$ of $Z$ such that $\omega|_{\partial\nN_0}=e^{-1}d\lambda_0$ (the neighborhood specified in Section~\ref{Main results}), so that there is a collar neighborhood $(-s_0,0]\times\partial\nN_0$ of $\partial\nN_0$ for which $\omega$ takes the form $d(e^{s-1}\lambda_0)$. Then we ``stretch the neck'' along this collar neighborhood, thereby inserting a compact piece of the symplectization $\left([-\kappa,-1]\times S^1\times S^2,d(e^s\lambda_0)\right)$ to $X$ while deforming $\omega$ (but keeping $\omega$ fixed on $X-\nN_0$), and apply the Moser trick (in the near-symplectic setting \cite{Honda:local}) to obtain our constrained neighborhood $\nN_\ast\subset X$ without deforming $\omega$. The second way is a slight variation of the first way. By \cite{Honda:local}*{Theorem 7} there is a deformation of $\omega$ on $X$ to a near-symplectic form that is ``standard'' in a neighborhood of $Z$. This standard local model $(S^1\times B^3,\omega_0)$ of a near-symplectic form near its zero-set is given in \cite{Taubes:HWZ}*{\S2e} (see also \cite{Honda:local}), which on a deleted neighborhood $S^1\times(B^3-\lbrace0\rbrace)$ diffeomorphic to $(-\infty,0]\times S^1\times S^2$ takes the form
$$\omega_0=-d\big(e^{s-\ln2}(1-3\cos^2\theta)dt+e^{\frac{3}{2}s-\ln\sqrt{6}}\sqrt{6}\cos\theta\sin^2\theta\,d\varphi\big)$$
This 2-form is not the symplectic form of a symplectization, but can be deformed so that on the region $[-\kappa,-1]\times S^1\times S^2$ it is $d(e^{s-1}\lambda_0)$, while it is not changed for both sufficiently small and sufficiently large $s\in(-\infty,0]$. Now the Moser trick \cite{Honda:local} is applied to obtain our constrained neighborhood $\nN_\ast\subset X$ without deforming $\omega$.
\end{proof}

\subsection{ECH cobordism maps}
\label{ECH cobordism maps}

\indent\indent
Fix $A\in\Rel(X)$. Thanks to Lemma~\ref{lem:nbhd}, we choose $\nN$ so that $(X_0,\omega)$ is a strong symplectic cobordism from the empty set $(\varnothing,0)$ to a disjoint union of $N$ copies of the contact 3-manifold $(S^1\times S^2,\lambda_A)$. Let $\overline X$ denote its completion, and fix a cobordism-admissible almost complex structure $J$ on $(\overline X,\omega)$. As shown in \cite{Hutchings:fieldtheory}, there are induced \textit{ECH cobordism maps} of the form
\begin{equation}
\label{eqn:chainmap}
ECH_0(\varnothing,0,0)\to ECH_\ast(\bigsqcup_{k=1}^NS^1\times S^2,\xi_0,1)
\end{equation}
defined by suitable counts of Seiberg-Witten instantons on $\overline X$. Since $ECH_0(\varnothing,0,0)\cong\ZZ$ is generated by the empty set of orbits, and a choice of ordering of the components of $\partial X_0$ defines an identification of $ECH_\ast(\bigsqcup_{k=1}^NS^1\times S^2,\xi_0,1)$ with $\bigotimes_{k=1}^NECH_\ast(S^1\times S^2,\xi_0,1)$, the map~\eqref{eqn:chainmap} should really be viewed as an element
$$\Phi_A\in \bigotimes_{k=1}^NECH_\ast(S^1\times S^2,\xi_0,1)$$
We now present a definition of $\Phi_A$ via counts of $J$-holomorphic curves in $\overline X$. In the sequel \cite{Gerig:Gromov} we will show that these two definitions coincide.

\begin{remark}
\label{I<0}
As explained in \cite{Hutchings:lectures,Hutchings:fieldtheory}, the main problem with constructing ECH cobordism maps via $J$-holomorphic curves is that negative ECH index curves can arise and it is currently unknown how to count them appropriately. In this paper, negative ECH index curves (with punctures) do not arise in the counts for $\Phi_A$ thanks to our choice of contact form for the boundary components of $X_0$. If we used $\lambda_0$ instead, the rotation classes of the exceptional orbits would allow the Conley-Zehnder index of their multiple covers to get large enough in magnitude to force the ECH index to become negative:

\begin{example}
Let $C\hookrightarrow\overline{X}$ be an embedded $J$-holomorphic plane asymptotic to the exceptional orbit $e_0$ with multiplicity 1, such that $\ind(C)=I(C)=0$. Fix the trivialization $\tau$ of $e_0^\ast\xi_0$ so that the rotation number of $e_0$ is $\sqrt{\frac{3}{2}}-1$. Then $c_\tau(C)=1$, $Q_\tau(C)=0$, $CZ_\tau(e_0)=1$, and
$$I(d\cdot C)=d\cdot c_\tau(C)+d^2Q_\tau(C)+CZ_\tau^I(d\cdot C)=-2\sum_{k=1}^d\lfloor k(\sqrt{\frac{3}{2}}-1)\rfloor$$
which is negative for $d\ge5$.
\end{example}
\end{remark}

If $X$ is not minimal, let $\eE_\omega\subset H_2(X_0,\partial X_0;\ZZ)$ denote the set of classes represented by \textit{symplectic} exceptional spheres in $X_0$. Note that there is a nondegenerate bilinear pairing
$$H_2(X_0;\ZZ)\otimes H_2(X_0,\partial X_0;\ZZ)\stackrel{\cdot}{\to}\ZZ$$
induced by Poincar\'e-Lefschetz duality and the relative cap-product, and $E\cdot E=-1$ for any $E\in\eE_\omega$. By intersection positivity of $J$-holomorphic curves, a $J$-holomorphic exceptional sphere is the unique $J$-holomorphic curve in its homology class.

In this case there is an addendum to Remark~\ref{I<0}. As with Taubes' Gromov invariants for closed symplectic 4-manifolds, the trouble with multiply covered exceptional spheres is that they have negative ECH index: a $J$-holomorphic exceptional sphere with multiplicity $d>1$ has index $-d(d-1)$. If $E\cdot A<-1$ for some $E\in\eE_\omega$, then it may be the case that $E$ is represented by a $J$-holomorphic exceptional sphere and $A$ is represented by a curve with an exceptional sphere component having multiplicity greater than one. In the next section we will see that this issue does not occur when $E\cdot A\ge-1$ for all $E\in\eE_\omega$ (such as when $X$ is minimal).

\begin{remark}
If $X_0$ is a closed symplectic 4-manifold (i.e. $\overline X=X_0=X$) then every class in $\eE_\omega$ has a $J$-holomorphic representative \cite{McDuff:rationalRuled}*{Lemma 3.1}, but the proof does not carry over in general when $\partial X_0\ne\varnothing$. The refinements to Taubes' Gromov invariants, introduced by McDuff \cite{McDuff:Gr}, thus do not always work for the near-symplectic Gromov invariants.

The reason we cannot repeat the proof of \cite{McDuff:rationalRuled}*{Lemma 3.1} for manifolds with boundary is the following. To argue that for generic cobordism-admissible $J$ on $\overline X$ any symplectic exceptional sphere is isotopic to a $J$-holomorphic one, we consider a parametrized moduli space of Fredholm index 0 spheres (representing a class in $\eE_\omega$) along a generic path of almost complex structures from a suitable reference almost complex structure to our desired $J$. We would need to establish compactness of this parametrized moduli space, i.e. that no sequence of holomorphic spheres breaks nor becomes nodal. But this cannot be guaranteed: a sequence of spheres can break at a non-generic almost complex structure, for which there can exist levels of the broken curve with Fredholm index $-1$, because the Fredholm index of a somewhere-injective curve is at least $-1$ (i.e. not necessarily $0$) for a generic path of almost complex structures. Breaking does not occur for closed 4-manifolds.
\end{remark}

\subsection{When there are no multiply covered exceptional spheres}
\label{When there are no multiply covered exceptional spheres}

\indent\indent
In this section we assume that either $X$ is minimal or $E\cdot A\ge-1$ for all $E\in\eE_\omega$. To construct the counts of $J$-holomorphic curves for Theorem~\ref{main theorem}, we first make the following choices: 

\bigskip
	$\bullet$ an integer $I\ge0$,
	
	$\bullet$ an integer $p\in\lbrace0,\ldots,I\rbrace$ such that $I-p$ is even,
	
	$\bullet$ an ordered set of $p$ disjoint oriented loops $\bar\eta:=\lbrace\eta_1,\ldots,\eta_p\rbrace\subset X_0$,
	
	$\bullet$ a set of $\frac{1}{2}(I-p)$ disjoint points $\bar z:=\lbrace z_1,\ldots,z_{(I-p)/2}\rbrace\subset X_0-\bar\eta$.

\bigskip	
\noindent Denote by $\mM_I(\varnothing,\Theta;A,\bar z,\bar\eta)$ the subset of elements in $\mM_I(\varnothing,\Theta)$ which represent the class $A$ and intersect all points $\bar z$ and all loops $\bar\eta$. Define the chain
\begin{equation}
\label{eqn:chain}
\Phi^I_{J,\bar z}(A,\bar\eta):=\sum_\Theta\sum_{\cC\in\mM_I(\varnothing,\Theta;A,\bar z,\bar\eta)}q(\cC)\cdot\Theta\in \bigotimes_{k=1}^NECC_\ast(S^1\times S^2,\lambda_A,1)
\end{equation}
where $\Theta$ indexes over the admissible orbit sets, and $q(\cC)\in\ZZ$ are weights that will be specified later in Section~\ref{Orientations and weights}. This is the chain which is to be used for Theorem~\ref{main theorem} and Definition~\ref{defn:Gr}.

The existence of curves in $\mM_I(\varnothing,\Theta;A,\bar z,\bar\eta)$ implies
$$\aA(\Theta)\le\rho(A)$$
so there are only finitely many orbit sets $\Theta$ which can arise in $\Phi^I_{J,\bar z}(A,\bar\eta)$. The fact that~\eqref{eqn:chain} is well-defined, for the given choices of data $\lbrace J,\bar z,\bar\eta\rbrace$, then follows from the following proposition.\footnote{Proposition~\ref{prop:finite} is the analog of \cite{Taubes:SWtoGr}*{Proposition 7.1} (also \cite{Taubes:counting}*{Proposition 4.3}) for closed symplectic manifolds and \cite{Hutchings:lectures}*{Lemma 5.10} for symplectizations.}

\begin{prop}
\label{prop:finite}
In the above setup, for generic $J$ and $\bar z\cup\bar\eta$, the moduli space $\mM_I(\varnothing,\Theta;A,\bar z,\bar\eta)$ is a finite set for each admissible orbit set $\Theta$.
\end{prop}

Before we prove this proposition we must figure out what the curves in $\mM_I(\varnothing,\Theta;A,\bar z,\bar\eta)$ look like. Given two distinct connected somewhere-injective curves $C_1$ and $C_2$, denote by $C_1\cdot C_2$ the algebraic count of intersections. In our analysis we will make use of a ``self-intersection'' number that depends on $\rho(A)$ and stems from \cite{Hutchings:beyond}*{Definition 4.7}. Given a connected somewhere-injective curve $C$, define
$$C\cdot C:=\frac{1}{2}\left[2g(C)-2+\ind(C)+h(C)+2e_A(C)+4\delta(C)\right]$$
where $g$ denotes its genus, $e_A$ denotes the total multiplicity of all $\rho(A)$-positive elliptic orbits, and $h$ denotes the number of ends at hyperbolic orbits. Note that when $C$ has no punctures, $C\cdot C$ is the algebraic count of self-intersections.

Next, it is crucial to know that for any multiply covered curve arising in our analysis, its Fredholm index is nonnegative. This is granted by the following lemma.

\begin{lemma}
\label{lem:coverFredholm}
Suppose $J$ is generic and $f:\tilde C\to C$ is a $d$-fold branched cover (with $b$ branch points) of a somewhere-injective $J$-holomorphic curve $C\to\overline X$, where the ends of $C$ are asymptotic to either positive hyperbolic orbits or $d$-positive elliptic orbits. Then $\ind(\tilde C)\ge b$. Furthermore, if $\tilde C$ is connected and $C$ is a Fredholm index 0 plane whose end is asymptotic to a $d$-positive elliptic orbit, then $\ind(\tilde C)>0$ unless $f=\1$.
\end{lemma}

\begin{proof}
The trivialization $\tau$ is assumed to be chosen so that the Conley-Zehnder indices of the orbits are as specified in Section~\ref{subsec:ECH}. Let $e$ (and $\tilde e$) denote the number of ends of $C$ (and $\tilde C$) at elliptic orbits. Then we compute
$$\ind(\tilde C)=b-d\cdot\chi(C)+2d\cdot c_\tau(C)-\tilde e=d\cdot\ind(C)+b+(de-\tilde e)$$
which is nonnegative because $\tilde e\le de$ and $\ind(C)\ge0$ by genericity of $J$.

If $C$ is the plane satisfying the hypotheses in the statement of the lemma (in particular, $e=1$), and $\tilde g$ denotes the genus of $\tilde C$, then the Riemann-Hurwitz formula implies
$$b=d+2\tilde g-2+\tilde e$$
and hence
$$\ind(\tilde C)=2(d+\tilde g-1)\ge0$$
which is zero if and only if $f=\1$.
\end{proof}

Finally, there will be certain pseudoholomorphic curves in $\overline X$ that require special attention. We give them a name:

\begin{definition}
A $J$-holomorphic curve in $\overline X$ is \textit{special} if it has Fredholm/ECH index zero, and is either an embedded torus or an embedded plane whose negative end is asymptotic to an embedded elliptic orbit with multiplicity one.
\end{definition}

\begin{prop}
\label{prop:moduli}
For generic $J$, every current in $\mM_I(\varnothing,\Theta;A,\bar z,\bar\eta)$ takes the following form. Its underlying components are embedded, pairwise disjoint, and pairwise do not both have negative ends at covers of the same elliptic orbit. An embedded component which intersects $l$ of the points in $\bar z$ and $l'$ of the loops in $\bar\eta$ has (ECH and Fredholm) index $2l+l'$. A component can be multiply covered only when it is special. In particular, the exceptional sphere components all have multiplicity one.
\end{prop}

\begin{proof}
Decompose a given current as
$$\cC=\lbrace(C_k,d_k)\rbrace\cup\lbrace(E_\sigma,m_\sigma)\rbrace$$
where $\sigma$ indexes the components which are (covers of) exceptional spheres. Then \cite{Hutchings:beyond}*{Proposition 4.8} implies
\begin{equation}
\label{eqn:inequality1}
\begin{split}
I= I(\cC)\ge& \sum_kd_k I(C_k)+\sum_kd_k(d_k-1)C_k\cdot C_k-\sum_\sigma m_\sigma(m_\sigma-1)\\
&+\sum_{k\ne k'}d_kd_{k'}\, C_k\cdot C_{k'}+\sum_{\sigma\ne \sigma'}m_\sigma m_{\sigma'}\, E_\sigma\cdot E_{\sigma'}+2\sum_{\sigma,k}m_\sigma d_k\,E_\sigma\cdot C_k
\end{split}
\end{equation}
Since $C_k$ is not an exceptional sphere and $\Theta$ contains only positive hyperbolic and $\rho(A)$-positive elliptic orbits,
\begin{equation}
\label{eqn:intersectBound}
C_k\cdot C_k\ge 2\delta(C_k)
\end{equation}
Indeed, $C_k$ cannot be a plane whose end is at a positive hyperbolic orbit because $\ind(C_k)$ would then be odd hence nonzero.

Since $A=\sum_kd_k[C_k]+\sum_\sigma m_\sigma [E_\sigma]$, the inequality~\eqref{eqn:inequality1} can be rewritten as 
\begin{equation}
\label{eqn:inequality2}
I\ge \sum_kd_k I(C_k)+2\sum_kd_k(d_k-1)\delta(C_k)+\sum_\sigma m_\sigma(m_\sigma+1)+\sum_{k\ne k'}d_kd_{k'}\, C_k\cdot C_{k'}+2\sum_\sigma m_\sigma[E_\sigma]\cdot A
\end{equation}
By the index inequality $\ind(C_k)\le I(C_k)-2\delta(C_k)$ and the assumption $E\cdot A\ge-1$ for all $E\in\eE_\omega$,
we can simplify the inequality further as
\begin{equation}
\label{eqn:inequality3}
I\ge \sum_kd_k\ind(C_k)+2\sum_kd^2_k\delta(C_k)+\sum_\sigma m_\sigma(m_\sigma-1)+\sum_{k\ne k'}d_kd_{k'}\, C_k\cdot C_{k'}
\end{equation}
If $C_k$ intersects $l\ge0$ of the base points and $l'\ge0$ of the loops then $\ind(C_k)\ge2l+l'$. Since $\cC$ contains the set $\bar z\cup\bar\eta$ (such that $I$ is the cardinality of $\bar\eta$ plus twice the cardinality of $\bar z$) we then have $\sum_k\ind(C_k)\ge I$, so in addition to the fact that each term on the right-hand-side of~\eqref{eqn:inequality3} is nonnegative we deduce that~\eqref{eqn:inequality3} is an equality,
\begin{equation}
\label{eqn:equality}
I=\sum_kd_k\ind(C_k)+2\sum_kd^2_k\delta(C_k)+\sum_\sigma m_\sigma(m_\sigma-1)+\sum_{k\ne k'}d_kd_{k'}\, C_k\cdot C_{k'}
\end{equation}
and that $\sum_k\ind(C_k)= I$. We conclude from both~\eqref{eqn:equality} and $\sum_k\ind(C_k)= I$ that
	
	$\bullet$ $d_k=1$ whenever $\ind(C_k)>0$,
		
	$\bullet$ $\ind(C_k)=2l+l'$ whenever $C_k$ intersects $l\ge0$ points in $\bar z$ and $l'\ge0$ loops in $\bar\eta$,
	 
	$\bullet$ $\delta(C_k)=0$ for all $k$, hence $I(C_k)=\ind(C_k)$,
		
	$\bullet$ $m_\sigma=1$ for all $\sigma$,
	 
	$\bullet$ $C_k\cdot C_{k'}=0$ for all $k\ne k'$.

\noindent
We then invoke all bullets (and our decomposition of $A$) to simplify~\eqref{eqn:inequality2} as the equality
$$0=2\sum_\sigma \sum_kd_k[E_\sigma]\cdot [C_k]+\sum_{\sigma\ne\sigma'} [E_\sigma]\cdot[E_{\sigma'}]$$
and conclude that $E_\sigma\cdot E_{\sigma'}=E_\sigma\cdot C_k=0$ for all $\sigma\ne\sigma'$ and all $k$. Note that the vanishing of these intersection numbers and the vanishing of the singularities imply that there are no nodal curves.\footnote{The existence of nodal curves is a priori possible because the symplectic form is not exact.} Similarly, the fact that~\eqref{eqn:inequality1} is now an equality implies, by \cite{Hutchings:beyond}*{Proposition 4.8}, that $C_k$ and $C_{k'}$ for $k\ne k'$ do not both have negative ends at covers of the same $\rho(A)$-positive elliptic orbit.

Finally, to specify the components which are multiply covered when $I(C_k)=0$, note that the equality~\eqref{eqn:inequality1} implies
$$0=C_k\cdot C_k$$
if $d_k>1$. From the definition of $C_k\cdot C_k$, a zero self-intersection can only occur if $C_k$ is a torus or an embedded plane with its end at an elliptic orbit with multiplicity one (or an embedded cylinder with ends at distinct hyperbolic orbits, but such cylinders cannot be multiply covered thanks to admissibility of $\Theta$).
\end{proof}

For the upcoming proof of Proposition~\ref{prop:finite}, it is crucial to know that a generic cobordism-admissible $J$ satisfies certain Fredholm regularity properties. We define these properties now.

\begin{definition}
\label{def:transversally}
Consider a somewhere-injective $J$-holomorphic curve $C$ and its deformation operator $D_C=\dbar+\nu_C+\mu_C$. Given any connected (branched) multiple cover $f:\tilde C\to C$ there is an induced pull-back deformation operator
$$D_f=\dbar+f^\ast\nu_C+f^\ast\mu_C$$
from $\Gamma(f^\ast N_C)$ to $\Gamma(T^{0,1}\tilde C\otimes f^\ast N_C)$. This Fredholm operator satisfies $\ind(D_f)=\ind(\tilde C)-2b$, where $b$ is the number of branch points of $f$ and $\ind(\tilde C)$ is the Fredholm index of the $J$-holomorphic curve $\tilde C\to C\to\overline X$ (see \cite{Wendl:automatic}*{Equation 3.18}).

For a given cobordism-admissible $J$, the moduli space $\mM_I(\varnothing,\Theta;A,\bar z,\bar\eta)$ is said to be \textit{cut out transversally} if the following is true for every current $\cC$ in the moduli space: For a component $(C,1)\in\cC$ its deformation operator $D_C$ is surjective, and for a component $(C,d)\in\cC$ with $d>1$ the pull-back $D_f$ of $D_C$ to any unbranched cover $f:\tilde C\to C$ with $\deg(f)\le d$ is injective.
\end{definition}

It turns out that regularity automatically holds for special planes (see Lemma~\ref{lemma:auto}). Likewise, Taubes was able to establish regularity for the special tori in his work on the Gromov invariants: The proof of \cite{Taubes:SWtoGr}*{Proposition 7.1} (see also the proof of \cite{Taubes:counting}*{Lemma 5.4}) explains how to perturb $J$ on special tori to achieve regularity, with one caveat. Taubes requires $J$ to be perturbed on the entire image of each torus, but the images of the tori appearing in $\mM_I(\varnothing,\Theta;A,\bar z,\bar\eta)$ can intersect the symplectization ends of $\overline X$, for which $J$ has additional constraints that do not appear in Taubes' work. Thankfully, this requirement of Taubes is weakened in \cite{Wendl:superrigidity}.\footnote{We point out that the deformation operators associated with holomorphic tori satisfy ``Petri's condition'' in the language of \cite{Wendl:superrigidity}, so the regularity statements in \cite{Wendl:superrigidity} have simpler proofs than for more general curves.} In particular, the perturbations can be made local to the curve in the cobordism region $X_0$.

\begin{proof}[Proof of Proposition~\ref{prop:finite}]
We follow the proof of \cite{Hutchings:lectures}*{Lemma 5.10} which argues that the ECH differential $\partial_\text{ECH}$ is well-defined. Suppose otherwise that there are infinitely many such currents in $\mM_I(\varnothing,\Theta;A,\bar z,\bar\eta)$. By Gromov compactness for currents, any sequence of currents in $\mM_I(\varnothing,\Theta;A,\bar z,\bar\eta)$ has a subsequence converging to a possibly broken ECH index $I$ current from $\varnothing$ to $\Theta$, with the cobordism level intersecting the points $\bar z$ and the loops $\bar\eta$. The symplectization levels of the broken holomorphic current do not have any closed components, thanks to exactness of the symplectic form (a sequence of tori sliding off the ends of $\overline X$ would have been a source of noncompactness). The cobordism level $\cC$ is an element of $\mM(\varnothing,\Theta';A',\bar z,\bar\eta)$ where $\Theta'$ is a potentially inadmissible orbit set that satisfies $\aA(\Theta')\le\rho(A)$, and $I\ge I(\cC)$. But the proof of Proposition~\ref{prop:moduli} shows that in fact $I=I(\cC)$ and that $\cC$ has no nodes. With nodes now excluded, no exceptional spheres arise in the limit if they did not originally exist.

To derive a contradiction from an infinite sequence of converging currents, we will make use of SFT compactness (for curves) in addition to Gromov compactness (for currents) as follows. We claim that there are only finitely many possibilities for the multiply covered components in our given sequence of currents. Assuming this claim for the moment, we can restrict to a subsequence $\lbrace\cC_\nu\rbrace\subset\mM_I(\varnothing,\Theta;A,\bar z,\bar\eta)$ so that the multiply covered components are the same. The remaining components of the subsequence are embedded and asymptotic to a fixed orbit subset $\Theta''\subseteq\Theta$, and by the above convergence as currents we can restrict to a further subsequence so that they represent the same relative homology class $A''\in H_2(\overline X,\varnothing,\Theta'')$ and have the same number of components. Now, $A''$ defines a ``topological complexity'' $J_0(A'')\in\ZZ$ (defined in \cite{Hutchings:revisited}*{\S6}), which by \cite{Hutchings:revisited}*{Proposition 6.14}\footnote{This proposition has an error in its statement, see the erratum to \cite{Hutchings:revisited}, but the error is vacuous in our scenario.} is an upper bound for the sum of the topological complexities of the embedded components $\lbrace C_{\nu,1},\ldots, C_{\nu,n}\rbrace$ of $\cC_\nu$,
$$\sum_{k=1}^nJ_0([C_{\nu,k}])\le J_0(A'')$$
Moreover, by \cite{Hutchings:revisited}*{Corollary 6.10} there is a genus bound on the topological complexities,
$$J_0([C_{\nu,k}])\ge 2g(C_{\nu,k})-2$$
It thus follows from these inequalities that there is a uniform bound on the genus of the embedded components, so we can restrict to a further subsequence so that the topological type of the embedded components is fixed. Now we invoke SFT compactness to obtain a convergent subsequence that converges as curves to a broken $J$-holomorphic curve.

By additivity of the ECH index, all symplectization levels have ECH index 0. Note that a sequence of closed components cannot break, because the only $I=0$ curves in a symplectization are covers of $\RR$-invariant cylinders. Using additivity of the Fredholm index in addition to nonnegativity of the Fredholm index by Lemma~\ref{lem:coverFredholm}, as in the proof of \cite{HutchingsTaubes:gluing1}*{Lemma 7.19}, we conclude that there is in fact only one level (that is, there cannot exist symplectization levels of the broken curve consisting solely of nontrivially branched covers of $\RR$-invariant cylinders). Thus, the limiting curve is unbroken.

To complete the argument that $\mM_I(\varnothing,\Theta;A,\bar z,\bar\eta)$ is a finite set, it remains to prove: 1) the aforementioned claim, and 2) that the limiting curve is isolated.

As for (1), if there were infinitely many possible multiply covered components (of a current representing the class $A$) then there would be infinitely many special curves (the underlying components). A formal repeat of the previous paragraphs shows that an infinite subsequence of such embedded curves must converge to an unbroken Fredholm index 0 curve. So it suffices to prove that this limiting curve is isolated, which is also the goal of (2).

As for (2), such isolation is only violated if a sequence of embedded curves $\lbrace C_k\rbrace$ converges to a multiply covered curve $f:\tilde C_\infty\to C_\infty$, where $C_\infty$ is a special curve. And this only occurs if $\ind(C_k)=\ind(\tilde C_\infty)=0$ for all $k$, otherwise there would be point constraints on $C_k$ and hence point constraints on $C_\infty$, contradicting the fact that $\ind(C_\infty)=0$. Thus, $\tilde C_\infty$ is an unbranched cover of $C_\infty$ by Lemma~\ref{lem:coverFredholm}. Now, such a convergent sequence $\lbrace C_k\rbrace$ would produce a nonzero element in the kernel of the pull-back $D_f$ of the deformation operator $D_{C_\infty}$, but this contradicts the fact that $\mM_I(\varnothing,\Theta;A,\bar z,\bar\eta)$ is cut out transversally for generic $J$.
\end{proof}

\subsection{Orientations and weights} 
\label{Orientations and weights}

\indent\indent
This section defines the integer weight $q(\cC)$ attached to a current $\cC=\lbrace(C_k,d_k)\rbrace$ which belongs to a given moduli space $\mM_I(\varnothing,\Theta;A,\bar z,\bar\eta)$, for a generic choice of $J$. This moduli space is assumed to be cut out transversally (see Definition~\ref{def:transversally}) and $E\cdot A\ge-1$ for all $E\in\eE_\omega$. The total weight is of the form
$$q(\cC)=\varepsilon(\cC)\prod_k r(C_k,d_k)\in\ZZ$$
where $r(C,d)$ is an integer weight attached to each component $(C,d)$, and $\varepsilon(\cC)$ is a global sign which depends on the ordering of both the set $\bar\eta$ and the set $\pP$ of positive hyperbolic orbits in $\Theta$. We specify these below.

\begin{remark}
\label{ordering}
Remember that $\Theta$ is an orbit set for the disjoint union of $N$ copies of $S^1\times S^2$. In order to identify ECH of this disjoint union with the tensor product of $N$ copies of ECH of $S^1\times S^2$, we need to choose an ordering of the $N$ copies of $S^1\times S^2$ (which is an ordering of the zero-circles of $\omega$). Once this ordering is made we can decompose $\pP$ as $\pP_1\sqcup\cdots\sqcup\pP_N$, where $\pP_k$ is the set of positive hyperbolic orbits in the orbit set $\Theta_k$ for the $k^\text{th}$ copy of $S^1\times S^2$, and $(\Theta,\fo)=\pm(\Theta_1,\fo_1)\otimes\cdots\otimes(\Theta_N,\fo_N)$. Here, the sign $\pm$ is determined by whether the ordering of $\pP$ defined by $\fo$ (dis)agrees with the ordering of $\pP_1\sqcup\cdots\sqcup\pP_N$ defined by $\fo_1\otimes\cdots\otimes\fo_N$.
\end{remark}

Given orientations of the admissible orbit sets ($\fo$ of $\Theta$), each moduli space $\mM_I(\varnothing,\Theta;A,\bar z,\bar\eta)$ is ``coherently oriented'' with the conventions in \cite{HutchingsTaubes:gluing2}*{\S9.5} (based off of \cite{BourgeoisMohnke}). In particular, $\ker D_C$ is oriented for every component $C$ of $\cC$.

We first define $r(C,1)$, following \cite{Taubes:counting}*{\S2} and \cite{HutchingsTaubes:weinsteinStable}*{\S2.5}. Recall that the set $\bar\eta$ is ordered, and labelled accordingly as $\lbrace\eta_1,\ldots,\eta_p\rbrace$. The curve $C$ intersects $l$ of the loops for some $l\in\lbrace0,\ldots,p\rbrace$, say $\lbrace\gamma_1,\ldots,\gamma_l\rbrace$ written in the order that they appear in $\bar\eta$. Denote the intersection points by $\lbrace w_1,\ldots,w_l\rbrace$. The curve $C$ then must intersect $\frac{1}{2}(I(C)-l)$ points in $\bar z$, which without loss of generality are the points $\lbrace z_1,\ldots,z_{(I(C)-l)/2}\rbrace$. From this data we build the $\RR^{I(C)}$-vector space
$$V_C:=\bigoplus^{\frac{1}{2}(I(C)-l)}_{i=1}N_{z_i}\oplus\bigoplus_{i=1}^{l}\big(N_{w_i}/\pi(T_{w_i}\gamma_i)\big)$$
where $N_x$ denotes the fiber of the normal bundle $N_C$ of $C$ over the point $x\in C$, and $\pi$ denotes the canonical projection of $T\overline X$ onto $N_C$. Then $V_C$ is oriented because the normal bundle of $C$ is oriented, the loops $\gamma_i$ are oriented, and the points $w_i$ are ordered; the ordering of the points $z_i$ does not matter because each $N_{z_i}$ is 2-dimensional. For generic $J$, the restriction map
$$\ker(D_C)\to V_C$$
is an isomorphism, and $r(C,1):=\pm1$ depending on whether this restriction map is orientation-preserving or orientation-reversing.

\bigskip
Before defining $r(C,d)$ when $d>1$, it will be useful to write out $r(C,1)$ in the cases where $I(C)=0$ and $C$ is either an exceptional sphere, a special torus, or a special plane. Then $r(C,1)$ is the modulo 2 count of spectral flow of a generic path of Fredholm first-order operators from the deformation operator $D_C$ to a complex linear operator. The following ``automatic transversality'' lemma will help us compute this spectral flow.

\begin{lemma}
\label{lemma:auto}
Let $C$ be a connected immersed $J$-holomorphic curve in $\overline X$ with ends at nondegenerate Reeb orbits $\lbrace\gamma_j\rbrace$, let $g$ denote the genus of $C$, and let $h_+$ denote the number of ends of $C$ at positive hyperbolic orbits (including even multiples of negative hyperbolic orbits). Given any Cauchy-Riemann type operator
$$D:L^2_1(N_C)\to L^2(T^{0,1}C\otimes N_C)$$
that is asymptotic to the fixed asymptotic operators $L_{\gamma_j}$ of $C$, if
$$2g-2+h_+<\ind(D)$$
then $D$ is surjective. In particular, if the inequality holds for the deformation operator $D_C$ then $C$ is transverse without any genericity assumption on $J$.
\end{lemma}

See \cite{Wendl:automatic}*{Proposition 2.2} for a proof of this lemma and for the technical definition of a ``Cauchy-Riemann type operator that is asymptotic to an asymptotic operator.'' In the case where $C$ is an exceptional sphere, Lemma~\ref{lemma:auto} implies $r(C,1)=1$. In the case where $C$ is a special torus, Lemma~\ref{lemma:auto} unfortunately does not say anything. In the case where $C$ is a special plane, its asymptotic operator over the elliptic orbit is $L$-flat and hence complex linear. Thus, its deformation operator has the form
$$D_C=\dbar+\nu_C+\mu_C$$
with complex anti-linear term $\mu_C$ asymptotic to zero along the end of $C$. We then apply Lemma~\ref{lemma:auto} to the path of Cauchy-Riemann type operators\footnote{The path of operators \eqref{eqn:path} might not arise from a path of cobordism-admissible $J$'s, which is why we have stated ``automatic transversality'' for operators that are not necessarily deformation operators of $J$-holomorphic curves.}
\begin{equation}
\label{eqn:path}
r\in[0,1]\mapsto\dbar+\nu_C+(1-r)\cdot \mu_C
\end{equation}
The asymptotic operators of~\eqref{eqn:path} are all the same, so the Fredholm index remains constant. Thus, there is no spectral flow when deforming $D_C$ to a complex linear operator and hence $r(C,1)=1$.

\bigskip
We now consider $r(C,d)$ with $d>1$. In the case where $C$ is a special torus, the explicit description of $r(C,d)$ is found in \cite{Taubes:counting}*{Definition 3.2} and will not be repeated here. Suffice to say, the weight assigned to a multiply covered torus is determined by the spectral flows of four operators on $C$ and guarantees that Taubes' Gromov invariants are well-defined. The reason is as follows: When deforming the data $(J,\omega)$, multiple covers of tori may ``pop off'' to nearby honest curves, or two tori with opposing sign may collide and annihilate, and this must be accounted for in order to obtain an actual invariant.

In the case where $C$ is a special plane, we define
$$r(C,d):=r(C,1)=1$$
because it turns out that $C$ is ``automatically super-rigid.'' To elaborate, similarly to the setup when $d=1$, we can deform any pull-back $D_f$ to a complex linear operator along a path of Cauchy-Riemann type operators such that the Fredholm index stays constant. The following lemma is a version of ``automatic transversality'' for such operators (see \cite{Wendl:automatic}*{Proposition 2.2} for a proof).

\begin{lemma}
\label{lemma:autoBranch}
Let $C$ be a connected immersed $J$-holomorphic curve in $\overline X$ with ends at nondegenerate Reeb orbits, let $f:\tilde C\to C$ be a branched cover with $b$ branch points, $\tilde g$ denote the genus of $\tilde C$, and let $\tilde h_+$ denote the number of ends of $\tilde C$ at positive hyperbolic orbits (including even multiples of negative hyperbolic orbits). Given any Cauchy-Riemann type operator
$$D:L^2_1(f^\ast N_C)\to L^2(T^{0,1}\tilde C\otimes f^\ast N_C)$$
that is asymptotic to fixed asymptotic operators of $\tilde C$ and satisfies $\ind(D)=\ind(\tilde C)-2b$, if
$$2\tilde g-2+\tilde h_++\ind(\tilde C)-2b<0$$
then $D$ is injective.
\end{lemma}

\begin{definition}
A connected somewhere-injective Fredholm index 0 curve $C$ is \textit{$d$-nondegenerate} if $\ker(D_f)=0$ for all unbranched covers $f$ of $C$ of degree no greater than $d$, and is \textit{super-rigid} if $\ker(D_f)=0$ for all (branched) covers $f$ of $C$.
\end{definition}

For a special plane $C$ and a given cover $f$,
\begin{equation}
\label{eqn:indexPlane}
\ind(\tilde C)=2-2\tilde g+2b-e-\sum_{i=1}^eCZ_\tau(\gamma_i)
\end{equation}
where $e$ is the number of punctures of $\tilde C$ such that its $i^\text{th}$ end is asymptotic to the elliptic orbit $\gamma_i$, and the trivialization $\tau$ is chosen so that the Conley-Zehnder indices are as specified in Section~\ref{subsec:ECH}. By plugging~\eqref{eqn:indexPlane} into Lemma~\ref{lemma:autoBranch}, we see that not only is $C$ super-rigid but there is no spectral flow when deforming $D_f$ to a complex linear operator for any $f$. This explains why $r(C,d)=1$.

\bigskip
It remains to define $\varepsilon(\cC)$. First, note that an even (respectively, odd) Fredholm index component of $\cC$ has an even (respectively, odd) number of ends asymptotic to positive hyperbolic orbits, and it intersects an even (respectively, odd) number of loops. Second, note that an ordering of the components of $\cC$ determines a partition of $\pP$ whose ordering differs from the fixed ordering of $\pP$ by a permutation $\sigma_\pP$, and it also determines a partition of $\bar\eta$ whose ordering differs from the fixed ordering of $\bar\eta$ by a permutation $\sigma_{\bar\eta}$. Third, note that any permutation $\sigma$ has a sign $\varepsilon(\sigma)$ given by its parity. Then
$$\varepsilon(\cC):=\varepsilon(\sigma_\pP)\varepsilon(\sigma_{\bar\eta})$$
which does not depend on the ordering of the components of $\cC$.

\begin{remark}
Our definition of $\varepsilon(\cC)$ is consistent with that in Taubes' construction of the Gromov invariants \cite{Taubes:counting}*{\S2}. There are no Reeb orbits to deal with for closed symplectic 4-manifolds, so $\varepsilon(\cC)$ reduces to the sign associated with the ordering of the set of loops $\bar\eta$. That sign is well-defined, i.e. independent of the ordering of the components of $\cC$, because every component of $\cC$ is closed and thus has even Fredholm index.
\end{remark}

\subsection{Equations for chain maps and chain homotopies}
\label{Equations for chain maps and chain homotopies}

\indent\indent
The goal of this section is to show that $\Phi^I_{J,\bar z}(A,\bar\eta)$ defines an element in ECH, and to clarify its dependence on $\bar z$ and $\bar\eta$. Again, we assume $E\cdot A\ge-1$ for all $E\in\eE_\omega$.

\begin{prop}
\label{prop:cycle}
For generic $J$, the chain $\Phi^I_{J,\bar z}(A,\bar\eta)$ is a cycle,
\begin{equation}
\label{chainmap}
\partial_{\op{ECH}}\circ\Phi^I_{J,\bar z}(A,\bar\eta)=0
\end{equation}
\end{prop}

The proof relies on a gluing theorem, for which we need some information about the asymptotics of the curves that are to be glued. The multiplicities of the positive (respectively, negative) ends of a connected curve asymptotic to a given orbit $\gamma$ defines a positive (respectively, negative) partition of the total multiplicity $m$ of $\gamma$. A key fact is that a nontrivial component of a holomorphic
current that contributes to the ECH differential or the $U$-map satisfies the \textit{partition conditions} \cite{Hutchings:lectures}*{\S3.9}, which means the partitions $p^\pm_\gamma(m)$ are uniquely determined by the orbit $\gamma$. For example, if $\gamma$ is elliptic with rotation class $\theta\in(0,\frac{1}{m})\mod 1$ then $p^+_\gamma(m)=(1,\ldots,1)$ and $p^-_\gamma(m)=(m)$.

\begin{proof}[Proof of Proposition~\ref{prop:cycle}]
To prove the chain map equation~\eqref{chainmap} we fix an admissible orbit set $\Theta$ and analyze the ends of $\mM_{I+1}(\varnothing,\Theta;A,\bar z,\bar\eta)$; here we clarify that the cardinality of $\bar z$ remains as $\frac{1}{2}(I-p)$ and the cardinality of $\bar\eta$ remains as $p\le I$. A broken curve arising as a limit of curves in $\mM_{I+1}(\varnothing,\Theta;A,\bar z,\bar\eta)$ cannot contain a multiply covered \textit{punctured} component in the cobordism level, except for a cover of a special plane or an embedded cylinder\footnote{\label{footnote:cylinders}There can exist multiple covers of index 0 embedded cylinders at hyperbolic orbits, because the orbit set between the levels of the broken curve need not be an \textit{admissible} orbit set.} (asymptotic to hyperbolic orbits). This follows from the arguments in the proof of Proposition~\ref{prop:finite}, for which the cobordism level $\cC$ now satisfies
$$I+1\ge I(\cC)\ge I$$

If there were no such multiply covered punctured components, the proof of $\partial_\text{ECH}^2=0$ in \cite{HutchingsTaubes:gluing1} could be copied to conclude that the number of signed gluings is 1 in the case needed to prove~\eqref{chainmap}. That is, the proof of \cite{HutchingsTaubes:gluing1}*{Lemma 7.23} would carry over verbatim to show that a broken curve consists of

	$\bullet$ a Fredholm/ECH index $I$ curve in the cobordism level (intersecting $\bar z$ and $\bar\eta$),
	
	$\bullet$ a Fredholm/ECH index 1 curve in the symplectization level, and
	
	$\bullet$ possibly additional levels between them consisting of \textit{connectors}, i.e.
	
	\indent~~~branched covers of $\RR$-invariant cylinders. 
	
\noindent The obstruction bundle ``gluing analysis'' used to prove \cite{HutchingsTaubes:gluing1}*{Theorem 7.20} would then prove~\eqref{chainmap}.

\begin{remark}
Multiple covers of exceptional spheres do not arise in this process: two curves having nonnegative index cannot glue together to form a negative index curve. This is consistent with the fact that the exceptional spheres are rigid and a sequence of such spheres do not break.
\end{remark}

It thus suffices to show that there is no further obstruction bundle gluing needed when our multiply covered curves are introduced. It also suffices to consider the broken limit of a sequence of Fredholm index $I+1$ curves, because in a convergent subsequence of our ECH index $I+1$ curves the multiply covered components (having ECH index 0) do not break, by the same argument in the proof of Proposition~\ref{prop:finite}. Now the multiple covers in the cobordism level must be unbranched, otherwise the Fredholm index would be too big (see Lemma~\ref{lem:coverFredholm}). Next, note that an unbranched cover of a plane (respectively, cylinder) is necessarily disjoint copies of a plane (respectively, cylinder), thanks to the Riemann-Hurwitz formula. Let's first analyze the cylinders, and then the planes:

For each $d>1$ we can compute the number of ways to glue a $d$-fold cover of a cylinder using the same reasoning as in the proof of $\partial_\text{ECH}^2=0$ which computes the number of gluings for two embedded ECH index 1 curves along hyperbolic orbits. In particular, it follows from \cite{HutchingsTaubes:gluing1}*{Lemma 1.7} that there are no connectors, otherwise their Fredholm index would be too big. Thus, the multiply covered cylinder would have to glue directly to a curve in the sympletization level, for which the positive partition conditions are satisfied. Subsequently, there are either no ways to glue or there are an even number of ways to glue, given by the number of permutations of the ends of the cylinders that glue (see \cite{HutchingsTaubes:gluing1}*{\S 1.5}). And as explained in \cite{HutchingsTaubes:gluing1}*{Remark 1.5}, half of those gluings have one sign and the other half have the other sign, so the total signed count of gluings is zero.

For each $d>1$ we claim that there is exactly one way to glue a $d$-fold cover of a special plane. If this multiply covered plane were to glue to an ECH index 1 curve with a connector inbetween, then the gluing of the multiply covered plane and the connector would be another (multiply covered) plane. But the only way this glued curve could have Fredholm index 0 is for the connector to be a disjoint union of $\RR$-invariant cylinders, thanks to Lemma~\ref{lem:coverFredholm}, so there is no nontrivial gluing. Thus, the multiply covered plane would have to glue directly to a curve in the symplectization level, for which the positive partition conditions are satisfied. That means the ends of the multiply covered plane satisfy $p^+_\gamma(d)=(1,\ldots,1)$, where $\gamma$ is the elliptic orbit which the underlying plane is asymptotic to, and thus there is exactly one possible gluing (the $d$ disjoint copies of the plane are indistinguishable, so there are no permutations for matching the ends of the curves).
\end{proof}

\bigskip
The cycle $\Phi^I_{J,\bar z}(A,\bar\eta)$ a priori depends on the choice of loops $\bar\eta$ and points $\bar z$ (and $J$). We now show that each loop in $\bar\eta$ can move around in its homotopy class and each point in $\bar z$ can move into the ends of $\overline X$, without affecting the induced homology class in ECH. Our argument follows \cite{HutchingsTaubes:weinsteinStable}*{\S 2.5} which argues that the $U$-map does not depend on the choice of base point.

\begin{prop}
\label{prop:U}
Fix an element $[\bar\eta]\in\Lambda^p(H_1(X;\ZZ)/\operatorname{Torsion})$ and an ordering of the zero-circles of $\omega$. For generic $J$, the homology class $Gr^I_{X,\omega,J}(A,[\bar\eta])$ represented by the chain $\Phi^I_{J,\bar z}(A,\bar\eta)$ does not depend on the choice of $\bar z$ nor on the choice of representative $\bar\eta\subset X_0$ for $[\bar\eta]$. Furthermore,
$$Gr^I_{X,\omega,J}(A,[\bar\eta])=U^{(I-p)/2}\circ Gr^p_{X,\omega,J}(A,[\bar\eta])\in \bigotimes_{k=1}^NECH_\ast(S^1\times S^2,\xi_0,1)$$
where $U^{(I-p)/2}$ denotes $\frac{1}{2}(I-p)$ compositions of any of the $U$-maps from components of $\partial X_0$.
\end{prop}

\begin{proof}
Let $\Theta$ be an admissible orbit set such that $\aA(\Theta)\le\rho(A)$. Let $\lbrace y_1,\ldots,y_{(I-p)/2}\rbrace\subset\partial X_0$ be a collection of base points which do not lie on any Reeb orbit. Pair each $y_k\in\partial X_0$ with $z_k\in\bar z$ and choose embedded paths $\gamma_k:[0,1]\to X_0$ from $z_k$ to $y_k$ such that the image of $\gamma_k$ only intersects $\partial X_0$ in $y_k$. Define the chains
$$K_{\gamma_k}:=\sum_\Theta\sum_{\cC\in\mM_{I-1}(\varnothing,\Theta;A,\bar z-z_k,\bar\eta,\gamma_k)}q(\cC)\cdot\Theta\in \bigotimes_{k=1}^NECC_\ast(S^1\times S^2,\lambda_A,1)$$
where $\Theta$ indexes over the admissible orbit sets. The proofs of Proposition~\ref{prop:finite} and Proposition~\ref{prop:moduli} also work in this setting to show that $K_{\gamma_k}$ is well-defined: the decrement $I\to I-1$ and the path constraint of $\gamma_k$ is compensated by the removal of the point constraint of $z_k$.

We can check that
\begin{equation}
\label{chainhomotopy}
\partial_\text{ECH}\circ K_{\gamma_k} = \Phi^I_{J,\bar z}(A,\bar\eta) - U_{y_k}\circ\Phi^{I-2}_{J,\bar z-z_k}(A,\bar\eta)
\end{equation}
by counting ends and boundary points of the moduli space $\mM_I(\varnothing,\Theta;A,\bar z-z_k,\bar\eta,\gamma_k)$, using the same gluing analysis as in the proof of Proposition~\ref{prop:cycle}. Passing to homology and iterating through the recursive equation~\eqref{chainhomotopy} for $1\le k\le \frac{1}{2}(I-p)$, the desired result concerning the $U$-map follows.

A similar argument yields the independence of the loop $\eta_k$ in $X_0$ that represents $[\bar\eta_k]$. Note here that $H_1(X;\ZZ)\cong H_1(X_0;\ZZ)$, which follows from the homological long exact sequence for the pair $(X,X_0)$ in addition to
$$H_k(X,X_0;\ZZ)\cong H_k(\operatorname{cl}\nN,\partial\nN;\ZZ)\cong H^{4-k}(\operatorname{cl}\nN;\ZZ)\cong H^{4-k}(Z;\ZZ)$$
These isomorphisms are given respectively by excision, Poincar\'e-Lefschetz duality, and $Z$ being a deformation retraction of $\operatorname{cl}\nN$; these homologies are trivial for $k\in\lbrace1,2\rbrace$.
\end{proof}

\subsection{Gradings} 
\label{Gradings}

\indent\indent
The remaining statement to be proved in Theorem~\ref{main theorem} is the fact that the element $Gr^I_{X,\omega,J}(A,[\bar\eta])$, despite being a sum over many orbit sets, is concentrated in a single grading of ECH.

\begin{lemma}
\label{lemma:gradings}
$Gr^I_{X,\omega,J}(A,[\bar\eta])\in ECH_{g(A,I)}(-\partial X_0,\xi_0,1)$, where the grading $g(A,I)$ is determined by $A$ and $I$. In terms of the canonical absolute $\ZZ/2\ZZ$ grading on ECH, the parity of $g(A,I)$ is equal to the parity of $I$.
\end{lemma}

\begin{proof}
We prove this in a slightly more general scenario. Consider a symplectic cobordism $(X,\omega)$ from the empty set to a contact 3-manifold $(Y,\lambda)$, a homology class $\Gamma\in H_1(Y;\ZZ)$ satisfying
$$c_1(\xi)+2\PD(\Gamma)=0$$
and a relative homology class $A\in H_2(X,-Y;\ZZ)$ satisfying $\partial A=-\Gamma$. Let $\alpha$ and $\beta$ be admissible orbit sets on $(Y,\lambda)$ in the class $\Gamma$. Let $\cC$ and $\cC'$ be $J$-holomorphic currents in the completion $\overline X$ which represent $A$ and are asymptotic to $\alpha$ and $\beta$, respectively. Since $\cC$ and $\cC'$ both represent $A$, the difference
$$[\cC']-[\cC]\in H_2(Y,\alpha,\beta)$$
can be used to measure the grading difference
$$|\alpha|-|\beta|=I([\cC']-[\cC])=I(\cC')-I(\cC)$$
where the last equality follows from additivity of the ECH index. Thus, $I(\cC')=I(\cC)$ if and only if $|\alpha|=|\beta|$.

For an admissible orbit set $\Theta$ on $(Y,\lambda)$, let $\e(\Theta)$ be the number of positive hyperbolic orbits in $\Theta$, which determines the parity of $|\Theta|$ as a canonical absolute $\ZZ/2\ZZ$ grading. The ``index parity'' formula \cite{Hutchings:lectures}*{Equation 3.7} states that
$$I(\cC)\equiv\e(\Theta)\mod2$$
for $\cC\in\mM_I(\varnothing,\Theta)$. Thus, the parities of $g(A,I)$ and $I$ agree. 
\end{proof}

As mentioned in Section~\ref{Main results}, a given spin-c structure $\s\in\Spinc(X)$ determines the relative class $A=A_\s\in H_2(X_0,\partial X_0;\ZZ)$ and also the index $I=d(\s)$. It will be shown in \cite{Gerig:Gromov} that $g(A_\s,d(\s))$ is equal to $N[\xi_\ast]$, i.e.
$$Gr^{d(\s)}_{X,\omega,J}(A_\s,[\bar\eta])\in\bigotimes_{k=1}^NECH_{[\xi_\ast]}(S^1\times S^2,\xi_0,1)\cong\ZZ$$
This fact, $g(A_\s,d(\s))=N[\xi_\ast]$, makes sense for a few reasons. First, $[\xi_\ast]$ has odd parity while $N$ has parity equal to that of $b^2_+(X)-b^1(X)+1$, which is also the parity of $d(\s)$, so the parity of $N[\xi_\ast]$ agrees with that of $d(\s)$ hence $g(A_\s,d(\s))$. Second, in the case that $N\ge2$ it follows from Proposition~\ref{prop:ECHSWF} that $ECH_j(-\partial X_0,\xi_0)$ is a single copy of $\ZZ$ if and only if $j=N[\xi_\ast]$.

\appendix
\section{Appendix: introducing twisted zero-circles}

\indent\indent
In this paper we have assumed that $Z$ consists of only untwisted zero-circles. As we now clarify, a straightforward modification allows us to include twisted zero-circles as long as they are non-contractible in $X$.

In the presence of a twisted zero-circle, the corresponding boundary component of $(X_0,\omega)$ is $(S^1\times S^2,e^{-1}\lambda^\sigma_0)$. Here, $\lambda^\sigma_0$ is the pushforward of $\lambda_0$ under the double covering map $S^1\times S^2\to S^1\times S^2$ for which the nontrivial deck transformation is the fixed-point-free involution
$$\sigma(t,\theta,\varphi)=(t+\pi,\pi-\theta,-\varphi)$$
Since $\lambda_0$ is $\sigma$-invariant, the orbits in $(S^1\times S^2,\lambda_0)$ descend to the orbits in $(S^1\times S^2,\lambda^\sigma_0)$, so that $\lambda^\sigma_0$ is a Morse-Bott contact form. The images of the exceptional orbits coincide, giving a single exceptional elliptic orbit with rotation class $\sqrt{\frac{3}{2}}\mod 1$. The images of the tori $T(\theta_0)$ for $\theta_0\ne\frac{\pi}{2}$ are also tori foliated by orbits. However, the image of the torus $T(\frac{\pi}{2})$ is a Klein bottle: it is described as a closed interval family of orbits, whose endpoints lift to orbits in $(S^1\times S^2,\lambda_0)$ that are fixed set-wise by $\sigma$. These endpoints are double covers of orbits that have one-half the period of the other orbits in the interval.

The contact structure $\xi^\sigma_0$ associated with $\lambda^\sigma_0$ is also overtwisted. Although $\xi_0$ and $\xi^\sigma_0$ are not homotopic over $S^1\times S^2$ \cite{Honda:local}*{Theorem 10}, they have the same Euler class: $e(\xi^\sigma_0)$ can be computed using the pushforward of the section $\sin\theta\,\partial_\theta\in\Gamma(\xi_0)$. Thus, $\xi_0$ and $\xi^\sigma_0$ are homotopic over the 2-skeleton of $S^1\times S^2$ and their corresponding spin-c structures are the same. Subsequently,
$$ECH_j(S^1\times S^2,\xi^\sigma_0,1)\cong \Hfrom^j(S^1\times S^2, \s_\xi+1)\cong ECH_j(S^1\times S^2,\xi_0,1)$$
and
$$J(S^1\times S^2,\s_\xi)=\big\{[\xi_0],[\xi^\sigma_0]\big\}$$
for which $[\xi_\ast]$ now has even parity under the absolute $\ZZ/2\ZZ$ grading on $ECH_\ast(S^1\times S^2,\xi^\sigma_0,1)$.

The analog of Lemma~\ref{lem:nbhd} now says that the relevant modification of $\lambda^\sigma_0$ for a given class $A\in\Rel(X)$ is a nondegenerate contact form $\lambda_A$ whose orbits of action less than $\rho(A)$ are either $\rho(A)$-positive elliptic, positive hyperbolic, \textit{or negative hyperbolic.} Indeed, Bourgeois' perturbation breaks up the Klein bottle of orbits into two doubly covered negative hyperbolic orbits and an embedded elliptic orbit with rotation number slightly positive \cite{Bourgeois:thesis}*{\S 2.2, \S9.5}, while the exceptional orbit of $\lambda^\sigma_0$ can be appropriately modified because the function~\eqref{eqn:f} constructed in the proof of Lemma~\ref{lem:f} is $\sigma$-invariant.

The existence of these negative hyperbolic orbits may cause problems. The problem with the proof of Proposition~\ref{prop:finite} for $\lambda^\sigma_0$ is that the inequality~\eqref{eqn:intersectBound} is false for an embedded plane $C$ asymptotic to a negative hyperbolic orbit. In particular, the current $(C,d)$ for $d>1$ has negative ECH index $-\frac{1}{2}d(d-1)$ and may arise in the cobordism level of a broken current. Of course, if the twisted zero-circles are non-contractible in $X$ then these planes cannot exist.


\end{document}